\documentclass[reqno]{amsart}
\usepackage[english]{babel}
\usepackage{amssymb,verbatim,cite}
\usepackage{enumerate}
\usepackage{needspace}
\usepackage[T1]{fontenc}
\newtheorem{theorem}{Theorem}[section]
\newtheorem{maintheorem}{Theorem}

\newtheorem{proposition}[theorem]{Proposition}
\newtheorem{lemma}[theorem]{Lemma}
\newtheorem{corollary}[theorem]{Corollary}
\theoremstyle{definition}
\newtheorem{definition}[theorem]{Definition}
\newtheorem{example}[theorem]{Example}
\newtheorem{question}[theorem]{Question}
\theoremstyle{remark}
\newtheorem*{remark}{Remark}
\numberwithin{equation}{section}

\newcommand{\CEP}[1]{\mbox{$\mathbb{C}^{#1}$}}

\newcommand{\N}{\mbox{$\mathbb{N}$}}

\newcommand{\PSH}[1]{\mbox{$\mathcal{PSH}(#1)$}}

\begin{document}

\title[MPSH Functions and the FS determinants in Hilbert spaces]{Maximal Plurisubharmonic Functions and Fujii--Seo Determinants in Hilbert Spaces}

\author{Per \AA hag}\address{Department of Mathematics and Mathematical Statistics\\ Ume\aa \ University\\SE-901~87 Ume\aa \\ Sweden}\email{per.ahag@umu.se}

\author{Rafa\l\ Czy{\.z}}\address{Faculty of Mathematics and Computer Science, Jagiellonian University,\newline \L ojasiewicza~6, 30-348 Krak\'ow, Poland}
\email{rafal.czyz@im.uj.edu.pl}

\author{Antti Per\"{a}l\"{a}}
\address{Department of Mathematics and Mathematical Statistics\\ Ume\aa \ University\\SE-901~87 Ume\aa \\ Sweden}
\email{antti.perala@umu.se}

\author{Jani Virtanen}
\address{University of Eastern Finland, Department of Physics and Mathematics, 80100 Joensuu, Finland. University of Reading, Reading, UK. University of Helsinki, Helsinki, Finland}
\email{jani.virtanen@uef.fi,\newline j.a.virtanen@reading.ac.uk, and \newline jani.virtanen@helsinki.fi}



\begin{abstract}
A maximal plurisubharmonic function on an
infinite-dimensional Hilbert space can have an injective, noncompact
Levi operator at every point. This raises the question of what
replaces the finite-dimensional determinant criterion for maximality.
We prove that a twice continuously differentiable plurisubharmonic
function is maximal if its Levi operator is compact at every point,
without any boundedness assumption on the function. Using the
normalized determinants of Fujii and Seo, we also characterize local
maximality by the condition that the infimum of the spectrum of the
Levi operator is zero. The same condition characterizes maximality
when the function is bounded on the closure of every bounded open
set whose closure lies in the domain. The compactness result follows
from a more general spectral criterion. Its proof combines a bounded
transformation with the smooth variational principle of Borwein and
Preiss to establish comparison on bounded open sets whose closures
need not be compact.
\end{abstract}

\maketitle

\section{Introduction}

In 1926, Riesz formulated the following characterization of subharmonic functions: if $u$ is upper semicontinuous and not identically $-\infty$ on a domain $\Omega\subset\mathbb R^2$, then $u$ is subharmonic if for every $\Omega'\Subset \Omega$ and every function $U$ that is harmonic on $\Omega'$ and continuous on $\overline{\Omega'}$, one has
\[
u\le U \ \text{on }\partial \Omega' \quad\Longrightarrow\quad u\le U \ \text{in }\Omega'.
\]
In this framework, harmonic functions may be viewed as the maximal subharmonic functions. For Riesz' original formulation, see~\cite{Rie26}.

Motivated by the same Riesz viewpoint, Sadullaev~\cite{sadu} in 1981 introduced maximal plurisubharmonic functions in several complex variables as follows: a plurisubharmonic function $u$ on a domain $\Omega\subseteq\CEP{n}$ is maximal if for every
$\omega\Subset\Omega$ and every upper semicontinuous function $v$ on $\bar\omega$ such that $v\in\PSH{\omega}$ and $v\le u$ on $\partial\omega$, one has $v\le u$ in $\omega$.
When $n=1$, maximal plurisubharmonic functions are precisely the harmonic functions. For background on maximal plurisubharmonic functions in $\CEP{n}$, see e.g.~\cite{Cegrell2009}.

In finite-dimensional complex analysis, maximality also admits an operator characterization. If $\Omega\subset\mathbb C^n$ and $u\in \PSH{\Omega}\cap C^2(\Omega)$, then the Levi form $D'D''u(z)$
is represented by a positive semidefinite Hermitian matrix. For $C^2$ plurisubharmonic functions, maximality is equivalent to the vanishing of the complex Monge--Amp\`ere measure $(dd^c u)^n$, or equivalently of the determinant of the Levi form, throughout the domain (see e.g.~\cite{Cegrell2004,Cegrell2009}).

Let $H$ be a nonzero complex Hilbert space. We take its inner product
to be linear in the first variable and conjugate-linear in the second.
For an open set $\Omega\subset H$, let $\mathcal{PSH}(\Omega)$ denote
the plurisubharmonic functions on $\Omega$. Their restrictions to
complex lines are subharmonic or identically $-\infty$.
If $u\in\mathcal{PSH}(\Omega)\cap C^2(\Omega)$, its Levi form is
represented by a bounded positive operator
$L_u(a)$ at each $a\in\Omega$. Our normalization is
$L_{\|\cdot\|^2}=I$. These definitions are recalled in Section~3.
For the Banach-space setting, see Mujica~\cite{M} and
Ligocka~\cite{Lig76}.

Maximal plurisubharmonic functions on Banach spaces were studied by
Dineen and Gaughran~\cite{DG93}. In the present paper, maximality
means comparison on every bounded open set $G$ with
$\overline G\subset\Omega$. More precisely, if
$v\in\mathcal{PSH}(G)$ and
\[
\limsup_{G\ni z\to\xi}v(z)\le u(\xi)
\qquad(\xi\in\partial G),
\]
then $v\le u$ in $G$. We call $u$ locally maximal if it is maximal
on a neighborhood of each point. The closures of these bounded
comparison sets need not be compact. This prevents the direct use
of a maximum argument on $\overline G$.

To mention some related work, Lempert~\cite{Lempert2004} proved that every real function locally bounded above on a pseudoconvex open subset of a Banach space with a countable unconditional basis has a majorant of the form $\|f\|$, where $f$ is a holomorphic map with values in a Banach space. This majorant is continuous and plurisubharmonic. In 2011, Patyi~\cite{Patyi2011} proved the existence of continuous plurisubharmonic majorants for locally bounded real functions on pseudoconvex open subsets of Banach spaces with a Schauder basis. Wang and Zhang~\cite{WangZhang2024} constructed plurisubharmonic exhaustion functions on pseudoconvex domains in $\ell^2$. Here we study the maximality of a prescribed twice continuously differentiable plurisubharmonic function through its Levi operator.

The finite-dimensional determinant of the Levi form is unavailable
for a general bounded positive operator on $H$. We seek a spectral
condition that characterizes local maximality and, under additional
hypotheses, gives comparison on all bounded open sets whose closures
lie in the domain.

There are several notions of determinants in infinite-dimensional
operator theory. For instance, Fredholm introduced a determinant in
his study of integral equations~\cite{Fre03}, and Fuglede--Kadison
defined a determinant for invertible elements of a $\mathrm{II}_1$
factor in terms of its canonical normalized trace~\cite{FK51,FK52}
(see also~\cite{Arv67,deLaHarpe13}). The Fredholm determinant applies,
in particular, to operators of the form $I+T$ with $T$ trace class. In contrast, the Fuglede--Kadison determinant is defined relative to a finite
trace on a von Neumann algebra. For the bounded positive Levi operators
considered here, we assume neither that they are trace-class
perturbations of the identity nor that they belong to a fixed finite
von Neumann algebra equipped with such a trace. This motivates
working with determinants attached to vector states instead.

Fujii and Seo~\cite{FS} introduced the normalized determinant
\[
\Delta_x(A)=\exp\langle(\log A)x,x\rangle,
\qquad \|x\|=1,
\]
for a bounded positive invertible operator $A$. For $A\ge0$ we use
the limit of $\Delta_x(A+\varepsilon I)$ as
$\varepsilon\searrow0$. Section~2 gives the equivalent definition
by a spectral integral, including the case in which that integral
equals $-\infty$. We define
\[
\operatorname{FSD}(u)(a)
=\inf_{\|x\|=1}\Delta_x(L_u(a)).
\]
Proposition~\ref{prop:extrema} shows that
\[
\operatorname{FSD}(u)(a)=\inf\sigma(L_u(a)),
\]
where $\sigma(A)$ denotes the spectrum of $A$.
Thus, the vanishing of this quantity means that the positive Levi
operator is not invertible. It does not require a nonzero vector in
its kernel.

\begin{maintheorem}\label{main:spectral}
Let $\Omega$ be an open subset of a nonzero complex Hilbert space
and let $u\in\mathcal{PSH}(\Omega)\cap C^2(\Omega)$. Then
\[
u\text{ is locally maximal in }\Omega
\quad\Longleftrightarrow\quad
\inf\sigma(L_u(a))=0\quad\text{for every }a\in\Omega.
\]
If $u$ is bounded above and below on $\overline G$ for every bounded
open set $G$ with $\overline G\subset\Omega$, these conditions are
also equivalent to maximality in $\Omega$.
\end{maintheorem}

Theorem~\ref{maximality_implies_zero},
Theorem~\ref{thm:bounded_converse}, and
Corollary~\ref{cor:local_maximality} prove Theorem~\ref{main:spectral}.
The sufficient condition uses the smooth variational principle of
Borwein and Preiss~\cite[Theorem~2.6]{BP87}. A quadratic perturbation
produces a point of contact at which the Levi form of a smooth test
function contradicts plurisubharmonicity of the competitor.

\begin{maintheorem}\label{main:compact}
Let $\Omega$ be an open subset of an infinite-dimensional complex
Hilbert space and let
$u\in\mathcal{PSH}(\Omega)\cap C^2(\Omega)$. If $L_u(a)$ is compact
for every $a\in\Omega$, then $u$ is maximal in $\Omega$.
\end{maintheorem}

Theorem~\ref{main:compact} follows from
Theorem~\ref{thm:hyperplane_maximality} and
Corollary~\ref{cor:pointwise_levi_maximality}. The theorem applies
more generally when
\[
\inf_{\substack{x\in E\\\|x\|=1}}
\langle L_u(a)x,x\rangle=0
\]
for every $a\in\Omega$ and every closed complex hyperplane
$E\subset H$. Here a hyperplane is a closed complex linear subspace
of codimension one. In the proof, the bounded function
$\arctan(v-u)$ allows the variational principle to be applied
without boundedness assumptions on $u$. The Levi form arising from
the subsequent composition has one additional term, which vanishes
on the kernel of a complex linear functional. This explains the
hyperplane condition. In particular, finite rank at each point
suffices, even when the ranks have no common finite bound.

We also prove maximality when, for every bounded open set $G$ with
$\overline G\subset\Omega$, there are unit vectors $x_n$ such that
\[
\sup_{z\in G}\langle L_u(z)x_n,x_n\rangle\longrightarrow0
\]
(Theorem~\ref{prop:approx_null_maximality}). This includes constant
Levi operators whose spectra have infimum zero. The condition is
distinct from compactness of each Levi operator. Indeed,
Example~\ref{e:rank_one_no_uniform_null} has Levi rank at most one
and fails this uniform condition. Example~\ref{e:unbounded_on_bounded_sets}
is maximal and unbounded above and below on a bounded ball.

We also prove comparison principles involving the individual
determinants. For positive invertible operators $A$ and $B$, the
inequalities $\Delta_x(A)\ge\Delta_x(B)$ for every unit vector $x$
are equivalent to $\log A\ge\log B$, which defines the
\emph{chaotic order}. Suppose that $u$ and $v$ are twice continuously
differentiable plurisubharmonic functions on a bounded domain and
are continuous on its closure, that $mI\le L_v\le MI$ for some
$0<m\le M$, and that $\Delta_x(L_u)\ge\Delta_x(L_v)$ at every
point and for every unit vector $x$. Then Theorem~\ref{cp_mixed}
gives a family of plurisubharmonic functions of the form
$c\,u+d_c\|z-z_*\|^2-v$, where $z_*\in H$ and the constants
$c\ge1$ and $d_c\ge0$ are given explicitly. The multiplicative and additive comparisons in
Corollary~\ref{cp1} are its two endpoints. In each case,
nonpositivity on the boundary implies nonpositivity in the domain.

Section~2 develops the spectral properties of $\Delta_x$ and the
operator inequalities needed later. Section~3 fixes the conventions
for plurisubharmonic functions and Levi forms. Section~4 proves the
maximality criteria, and Section~5 gives examples distinguishing
their hypotheses. Section~6 applies arithmetic-geometric mean
estimates and their reverses to comparison of plurisubharmonic
functions. Section~\ref{sec:open} states the remaining questions
about local maximality and functions of lower regularity.

\section{The Fujii--Seo determinant}

Throughout this section, $H$ is a nonzero complex Hilbert space and
$\mathcal B(H)$ denotes the algebra of bounded linear operators on $H$.
We write $A\ge0$ if $A$ is self-adjoint and
$\langle Ax,x\rangle\ge0$ for all $x\in H$, and $A>0$ if $A\ge0$
is invertible. Thus $A>0$ is equivalent to $A\ge mI$ for some $m>0$.
The spectrum of $A$ is denoted by $\sigma(A)$.

We recall the Fujii--Seo determinant $\Delta_x(A)$ introduced
in~\cite{FS} for a positive invertible operator $A$ and a unit vector
$x$. We then extend the definition to all positive operators and
establish the spectral identities used in the maximality criteria.
The final estimates will be applied to Levi operators in the
comparison principles.

If $A>0$ is a bounded positive invertible operator and $x\in H$ is a unit vector, then $\log A$ is a bounded self-adjoint operator, and one defines
\begin{equation}\label{eq:def_delta}
\Delta_x(A):=\exp\bigl(\langle (\log A)x,x\rangle\bigr).
\end{equation}
For $A\ge 0$ not necessarily invertible, we define the notation $\langle (\log A)x,x\rangle$ as an extended real number by regularization
\begin{align}\label{eq:def_delta_semidef}
\langle (\log A)x,x\rangle &:=\lim_{\varepsilon\to 0^+}\langle \log(A+\varepsilon I)x,x\rangle\in[-\infty,\infty),\\\notag
\Delta_x(A) &:=\exp\bigl(\langle (\log A)x,x\rangle\bigr),
\end{align}
with the convention $\exp(-\infty)=0$.
The limit exists because the integrands in the spectral calculus
decrease as $\varepsilon\searrow0$.
If $A=\int_{\sigma(A)}\lambda\,dE(\lambda)$ is the spectral resolution
of $A$ and $E_x(B):=\langle E(B)x,x\rangle$ for Borel sets $B$, then
$E_x$ is a probability measure and
\[
\langle (\log A)x,x\rangle
=\int_{\sigma(A)}\log\lambda\,dE_x(\lambda),
\qquad \log0:=-\infty.
\]
The integral is well defined in $[-\infty,\infty)$ because
$\log\lambda$ is bounded above on the spectrum. In particular,
$\Delta_x(A)\in[0,\|A\|]$. For a noninvertible operator, the notation
in~\eqref{eq:def_delta_semidef} does not assert that $\log A$ is a
bounded operator.

\begin{proposition}\label{basic_properties}
Fix a unit vector $x\in H$.
\begin{enumerate}\itemsep2mm
\item If $A>0$, then the map $A\mapsto \Delta_x(A)$ is norm continuous. (\cite[Section~2]{FS})

\item If $A>0$, then
\[
\langle A^{-1}x,x\rangle^{-1}\le \Delta_x(A)\le \langle Ax,x\rangle.
\] (\cite[Theorem~2]{FS})

\item If $A>0$, then
\[
\|A^{-1}\|^{-1}\le \Delta_x(A)\le r(A)=\|A\|,
\] where $r(A)$ denotes the spectral radius of $A$. (\cite[Corollary~3]{FS})

\item If $0<mI\le A\le MI$, then the following reverse inequality holds. When $m<M$,
\[
\langle Ax,x\rangle\le a\,\exp\!\left(\frac{b-a}{a}\right)\Delta_x(A)
= S\!\left(\frac Mm\right)\Delta_x(A),
\]
where
\[
a:=\frac{M-m}{\log M-\log m},
\qquad
b:=\frac{m\log M-M\log m}{\log M-\log m},
\]
and Specht's ratio is defined for $h>0$ by
\[
S(h):=\frac{(h-1)h^{1/(h-1)}}{e\log h}\qquad (h\neq 1),\qquad S(1):=1.
\]
If $m=M$, then $A=mI$ and
\[
\langle Ax,x\rangle=\Delta_x(A)=m,
\]
so the inequality holds with equality. (\cite{FIS})
\end{enumerate}
\end{proposition}

By applying Proposition~\ref{basic_properties}(2) to $A+\varepsilon I$
and letting $\varepsilon\searrow0$, we obtain
\[
0\le\Delta_x(A)\le\langle Ax,x\rangle\qquad(A\ge0).
\]
The definition also gives
\[
\Delta_x(tA)=t\Delta_x(A)\quad(t>0),
\qquad
\Delta_x(A^p)=\Delta_x(A)^p\quad(p>0).
\]
For $A>0$, the latter identity holds for every $p\in\mathbb R$.
Moreover, if $0\le A\le B$, operator monotonicity of the logarithm
and regularization give $\Delta_x(A)\le\Delta_x(B)$.

\begin{remark}
The upper bound in Proposition~\ref{basic_properties} $(2)$ is the arithmetic-geometric mean inequality, and $(4)$ is its reverse.
In (2), $\Delta_x(A)=\langle Ax,x\rangle$ holds if and only if $x$ is an eigenvector of $A$ (see \cite[Introduction]{FIS}).
For the equality statement in $(4)$, assume $m<M$. Then equality is more rigid, forcing $m$ and $M$ to be eigenvalues of $A$, and $x$ to be a specific linear combination of the corresponding eigenvectors (see \cite{FIS}). If $m=M$, then $A=mI$ and equality holds for every unit vector $x$.
\end{remark}

If $\dim H=n<\infty$, then for a positive definite matrix $A$ we have
\[
\Delta_x(A)=\prod_{j=1}^{n}\lambda_j^{y_j},
\qquad
y_j:=\langle E_jx,x\rangle\geq 0,\ \ \sum_{j=1}^n y_j=1,
\]
where $A=\sum_{j=1}^{n}\lambda_j E_j$ is a spectral decomposition into rank-one orthogonal projections $E_j$, and the eigenvalues $\lambda_j$ are counted with multiplicity.
Thus, $\Delta_x(A)$ is a \emph{weighted geometric mean} of the eigenvalues of $A$.
In particular, $\Delta_x(A)$ generalizes $(\det A)^{1/n}$: if we choose $x$ so that $y_j=1/n$ for all $j$, then $\Delta_x(A)=(\det A)^{1/n}$.

The relation with power means also holds in infinite dimensions.
For $A>0$ and $\|x\|=1$,
\[
\lim_{p\to0}\langle A^p x,x\rangle^{1/p}=\Delta_x(A).
\]
The convergence is decreasing as $p\searrow0$ and increasing as
$p\nearrow0$ \cite[Theorem~4]{FS}.
If $A,B>0$ commute, then $AB>0$ and
\[
\Delta_x(AB)=\Delta_x(A)\Delta_x(B),
\]
since $\log(AB)=\log A+\log B$.

\begin{proposition}\label{prop:extrema}
Let $A\ge0$. Then
\begin{align*}
\inf_{\|x\|=1}\Delta_x(A)
&=\inf_{\|x\|=1}\langle Ax,x\rangle=\inf\sigma(A),\\
\sup_{\|x\|=1}\Delta_x(A)
&=\sup_{\|x\|=1}\langle Ax,x\rangle=\sup\sigma(A).
\end{align*}
\end{proposition}
\begin{proof}
The equalities involving $\langle Ax,x\rangle$ follow from the spectral
theorem. If $A>0$, apply the same identities to the bounded
self-adjoint operator $\log A$ and use the spectral mapping theorem.
This gives both equalities for $\Delta_x(A)$.

Suppose that $A$ is not invertible. Then $\inf\sigma(A)=0$, and
$0\le\Delta_x(A)\le\langle Ax,x\rangle$ gives the identity for the infimum. The case $A=0$ is immediate. If $A\ne0$ and
$0<t<\|A\|$, the spectral projection $E((t,\|A\|])$ is nonzero.
Choose a unit vector $x$ in its range. The spectral measure $E_x$
is supported in $(t,\|A\|]$, so
\[
\Delta_x(A)
=\exp\!\left(\int_{(t,\|A\|]}\log\lambda\,dE_x(\lambda)\right)
\ge t.
\]
Since $\Delta_x(A)\le\|A\|$ for every unit vector $x$, letting
$t\nearrow\|A\|$ proves the identity for the supremum.
\end{proof}

\begin{proposition}\label{bp2}
Let $A\ge0$. The following are equivalent.
\begin{enumerate}\itemsep2mm
\item $\inf_{\|x\|=1}\Delta_x(A)=0$.
\item $\inf_{\|x\|=1}\langle Ax,x\rangle=0$.
\item There exists a sequence $(x_n)$ of unit vectors with $\langle Ax_n,x_n\rangle\to0$.
\item There exists a sequence $(x_n)$ of unit vectors with $\|Ax_n\|\to0$.
\item $0\in\sigma(A)$.
\item $A$ is not invertible in $\mathcal B(H)$.
\end{enumerate}
\end{proposition}
\begin{proof}
Proposition~\ref{prop:extrema} and the definition of the spectrum give
the equivalence of \textup{(1)}, \textup{(2)}, \textup{(5)}, and
\textup{(6)}. The equivalence of \textup{(2)} and \textup{(3)} follows
from the definition of the infimum. Finally, for every unit vector $x$,
\[
\|Ax\|^2\le\|A\|\langle Ax,x\rangle,
\qquad
\langle Ax,x\rangle\le\|Ax\|.
\]
These inequalities give the equivalence of \textup{(3)} and \textup{(4)}.
\end{proof}

\begin{remark}
If $A\ge0$ and $\inf\sigma(A)=0$, there is a unit vector $x$ with
$\Delta_x(A)=0$, even when $A$ is injective.
Indeed, if $\ker A\ne\{0\}$, choose a unit vector $x\in\ker A$.
Otherwise, zero is an accumulation point of the positive spectrum.
Choose $\lambda_n\in\sigma(A)\cap(0,\infty)$ so that
$\lambda_n<e^{-n^3}/2$ and $\lambda_{n+1}<\lambda_n/4$.
The intervals $I_n=(\lambda_n/2,3\lambda_n/2)$ are pairwise disjoint,
and their spectral projections are nonzero. Choose unit vectors
$e_n\in\operatorname{Ran}E(I_n)$ and set
\[
x=c\sum_{n=1}^\infty\frac{e_n}{n},\qquad
c=\left(\sum_{n=1}^\infty n^{-2}\right)^{-1/2}.
\]
Orthogonality of the spectral subspaces gives $\|x\|=1$ and
$E_x=c^2\sum_n n^{-2}E_{e_n}$. Hence
\[
\int(-\log\lambda)\,dE_x(\lambda)
\ge c^2\sum_{n=1}^\infty n^{-2}n^3=\infty,
\]
so $\Delta_x(A)=0$. For an injective $A$, this vector still satisfies
$\langle Ax,x\rangle>0$.
If $\inf\sigma(A)>0$, attainment of the lower endpoint instead requires
an eigenvector at that endpoint. For $A\ne0$, attainment of the upper
endpoint likewise requires an eigenvector at $\|A\|$.
\end{remark}

Recall that the \emph{chaotic order} is defined for invertible positive operators by
\[
A\gg B
\quad\Longleftrightarrow\quad
\log A\ge \log B.
\]
This order appears frequently in operator inequalities. Compare also Fujii and Seo~\cite{FSchaotic}, who characterize the chaotic order by related additive operator inequalities derived from determinant estimates.

The next observation shows that it is the same order as that detected by $\Delta_x$.

\begin{proposition}\label{log}
Let $A, B>0$. Then $A\gg B$ if and only if $\Delta_x(A)\ge \Delta_x(B)$ for all unit vectors $x\in H$.
\end{proposition}

\begin{proof}
Since $\log A$ and $\log B$ are bounded self-adjoint operators, the condition $\log A\ge \log B$ is equivalent to $\langle (\log A)x,x\rangle\ge \langle (\log B)x,x\rangle$ for all unit vectors $x$, which in turn is equivalent to $\Delta_x(A)\ge \Delta_x(B)$ for all unit vectors $x$.
\end{proof}

\begin{remark}
Operator monotonicity of the logarithm shows that $A\ge B>0$
implies $A\gg B$. If $A$ and $B$ commute, the converse follows by
applying the increasing function $t\mapsto e^t$ in the commutative
$C^*$-algebra generated by $A$ and $B$. Thus the two orders agree for
commuting positive invertible operators.
\end{remark}

However, the implication $B\gg A\Rightarrow B\ge A$ does not hold in general, as the following example shows. This example was very kindly provided to us by Piotr Niemiec.

\begin{example}
Let
\[
A=
\begin{pmatrix}
1&0\\0&4
\end{pmatrix}
\qquad\text{and}\qquad
B=
\begin{pmatrix}
5&5\\5&10
\end{pmatrix}.
\]
Then $\log A\le \log B$ but $B-A$ is not positive semidefinite.
Indeed, $\det(B-A)=-1$. To verify the logarithmic inequality, put
\[
\varphi:=\frac{1+\sqrt5}{2},\qquad
k:=\frac{2\log\varphi}{\sqrt5},\qquad d:=\log(5/2).
\]
Direct diagonalization gives
\[
\log B-\log A
=dI+\begin{pmatrix}\log2-k&2k\\2k&k-\log2\end{pmatrix},
\]
whose eigenvalues are $d\pm\sqrt{(\log2-k)^2+4k^2}$.
The elementary bounds
\[
d>0.91629,\qquad 0.69314<\log2<0.69315,\qquad
0.43040<k<0.43042
\]
give
\begin{align*}
d^2-(\log2-k)^2-4k^2
&>0.91629^2-0.26275^2-0.86084^2\\
&=0.029504296>0.
\end{align*}
These bounds follow, for $y>1$, from
\[
\log y=2\sum_{j=0}^N\frac{t^{2j+1}}{2j+1}+R_N,
\qquad t=\frac{y-1}{y+1},
\]
where
\[
0<R_N\le\frac{2t^{2N+3}}{(2N+3)(1-t^2)}
.
\]
Thus both eigenvalues are positive, and $\log B-\log A>0$.
\hfill{$\Box$}
\end{example}

The following additive estimate is due to Fujii and Seo
\cite[Theorem~10]{FS}. The formula below corrects a sign misprint in
the constant printed there. We include a proof to specify the
constant used in the comparison principles.

\begin{proposition}\label{prop:additive_reverse}
Assume $0<mI\le A\le MI$.
Then for each unit vector $x\in H$,
\[
0\le \langle Ax,x\rangle-\Delta_x(A)\le C(m,M),
\]
where
\[
C(m,M):=
\begin{cases}
\frac{M-m}{\log M-\log m}\,\log S\!\left(\frac Mm\right) & \text{if } m<M,\\[2mm]
0 & \text{if } m=M.
\end{cases}
\]
\end{proposition}

\begin{proof}
If $m=M$, then $A=mI$ and the conclusion is immediate. Suppose $m<M$,
and let $a,b$ be as in Proposition~\ref{basic_properties}(4).
Jensen's inequality gives $\Delta_x(A)\le\langle Ax,x\rangle$.
The chord inequality $e^s\le as+b$ on $[\log m,\log M]$
gives $A\le a\log A+bI$. Thus, with $t=\langle Ax,x\rangle$,
\[
\Delta_x(A)\ge \exp\bigl((t-b)/a\bigr).
\]
The maximum over $t\in\mathbb R$ of $t-\exp((t-b)/a)$ is
$a\log a+b-a$. Consequently,
\[
\langle Ax,x\rangle-\Delta_x(A)
\le a\log a+b-a=a\log S(M/m)=C(m,M).
\]
\end{proof}

We next combine an additive term with a multiplicative constant.
The inequality in the following proposition is a special case of
\cite[Theorem~8.14]{PFMJS05}, applied to $p\log A$ and $p\log B$
after adding a common real multiple of $I$ if necessary, with both
functions equal to the exponential. We give a direct proof and
verify the optimality of the additive constant.

\begin{proposition}\label{prop:mixed_add_mult}
Let $A,B>0$, suppose that $A\gg B$ and $mI\le B\le MI$,
where $0<m<M$, and fix $p>0$. Put
\[
L:=\frac{M^p-m^p}{\log M^p-\log m^p},\qquad
S:=S((M/m)^p).
\]
For every $c\in[1,S]$,
\[
B^p\le cA^p+L\log(S/c)I.
\]
If $\dim H\ge2$, then $L\log(S/c)$ is the smallest nonnegative
constant for which this inequality holds for all such $A,B$.
If $\dim H=1$, the smallest nonnegative constant is zero.
\end{proposition}
\begin{proof}
Set $a=m^p$, $b=M^p$, $\alpha=\log a$, $\beta=\log b$,
and $\eta=a-L\alpha$. Then $S=L\exp(\eta/L-1)$.
Write $X=p\log A$ and $Y=p\log B$. We have
$X\ge Y$ and $\alpha I\le Y\le\beta I$. The chord inequality
for the exponential gives
\[
e^Y\le LY+\eta I\le LX+\eta I.
\]
For every $t\in\mathbb R$,
\[
Lt+\eta-ce^t\le L\log(L/c)+\eta-L=L\log(S/c),
\]
since the maximum is attained at $t=\log(L/c)$. Functional calculus
for $X$ proves the inequality.

For optimality in dimension two, set $t=\log(L/c)$.
Since $a<L<b$, one has $\log L\in(\alpha,\beta)$.
Also
\[
\log(L/S)=1-\eta/L
=\alpha+1-a/L=\beta+1-b/L\in(\alpha,\beta).
\]
Thus $t\in(\alpha,\beta)$ for every $c\in[1,S]$. Put
\[
q=\sqrt{(t-\alpha)(\beta-t)},\qquad s=\alpha+\beta-t,
\]
and, for $\varepsilon>0$, define
\[
Y=\begin{pmatrix}t&q\\q&s\end{pmatrix},\qquad
X_\varepsilon=\begin{pmatrix}t+\varepsilon&0\\0&s+q^2/\varepsilon\end{pmatrix}.
\]
The spectrum of $Y$ is $\{\alpha,\beta\}$, and
\[
X_\varepsilon-Y=
\begin{pmatrix}\varepsilon&-q\\-q&q^2/\varepsilon\end{pmatrix}\ge0.
\]
Consequently $A_\varepsilon=\exp(X_\varepsilon/p)$ and
$B=\exp(Y/p)$ satisfy all the hypotheses. Since $e^Y=LY+\eta I$,
\[
\langle(B^p-cA_\varepsilon^p)e_1,e_1\rangle
=Lt+\eta-ce^{t+\varepsilon}
\longrightarrow L\log(S/c).
\]
No smaller nonnegative constant works. On any larger Hilbert space,
take the orthogonal direct sum with $A=B=mI$ on the complement of a
two-dimensional subspace. In dimension one, $A\gg B$ implies $A\ge B$,
so $B^p\le cA^p$ for $c\ge1$.
\end{proof}

\begin{remark}
With the notation of Proposition~\ref{prop:mixed_add_mult}, the
function $c\mapsto L\log(S/c)$ is strictly convex. For $1<c<S$,
it is therefore strictly smaller than
$\frac{S-c}{S-1}L\log S$. The endpoint $c=1$ gives the optimal
nonnegative additive constant $L\log S$ in dimension at least two.
The endpoint $c=S$ gives $B^p\le SA^p$, and the optimality for
$c<S$ shows that $S$ is the smallest multiplicative constant in that
dimension. If $m=M$, then $B=mI$ and $A\gg B$ implies $A\ge mI$.
In this case $B^p\le A^p$ for every $p>0$.
\end{remark}

\section{Plurisubharmonic functions}

We recall the definitions and basic properties of plurisubharmonic
functions on a Hilbert space, fix the normalization of the Levi form,
and prove a boundary maximum principle used in the maximality and
comparison arguments.

\begin{definition}
Let $\Omega$ be an open set of the Hilbert space $H$. A function $u:\Omega\to [-\infty,+\infty)$ (not identically equal to $-\infty$ on any component of $\Omega$) is said to be {\it plurisubharmonic} if $u$ is upper semicontinuous and for each $z\in \Omega$ and $h\in H$ such that $z+\lambda h\in \Omega$ for $|\lambda|\leq 1$, one has
\[
u(z)\leq \frac {1}{2\pi}\int_0^{2\pi}u(z+e^{i\theta}h)\,d\theta.
\]
By $\mathcal{PSH}(\Omega)$ we denote the family of plurisubharmonic functions on $\Omega$.
\end{definition}

Plurisubharmonicity is a local property. The integral above is well-defined (possibly equal to $-\infty$).

Whenever a comparison set is used, we state explicitly that it is
bounded and open and that its closure lies in the domain.
In infinite dimensions this does not imply compactness of the closure.

We recall some basic properties of plurisubharmonic functions (see~\cite{M}). For the Levi form in the Banach-space setting, compare also Ligocka~\cite{Lig76}.

For real-valued functions, $C^2(\Omega)$ means twice continuously
Fr\'echet differentiable on the underlying real Hilbert space.
In particular, the second differential and the Levi operator depend
continuously on the base point in operator norm.

Recall that if $Df(a)$ denotes the real differential of $f$ at $a$, then
\[
D'f(a)(h)=\frac 12\left(Df(a)(h)-iDf(a)(ih)\right)
\]
and
\[
D''f(a)(h)=\frac 12\left(Df(a)(h)+iDf(a)(ih)\right).
\]
Moreover,
\[
4D'D''f(a)(s,t)=D^2f(a)(s,t)+D^2f(a)(is,it)+iD^2f(a)(s,it)-iD^2f(a)(is,t).
\]

For real-valued $u\in C^2(\Omega)$, the Levi form $D'D''u(a)$ is a
continuous Hermitian form, linear in its first argument and
conjugate-linear in its second. We denote by $L_u(a)$ the unique bounded
self-adjoint operator satisfying
\[
D'D''u(a)(s,t)=\langle L_u(a)s,t\rangle
\qquad(s,t\in H).
\]
We call $L_u(a)$ the Levi operator of $u$ at $a$. In particular,
\[
\Delta_t u(a+th)\big|_{t=0}=4\langle L_u(a)h,h\rangle,
\qquad L_{\|\cdot\|^2}(a)=I.
\]

\begin{theorem}\label{thm_basicprop1}
Assume that $\Omega \subset H$ is an open set. We use the convention $0\cdot(-\infty)=0$. Then:
\begin{enumerate}\itemsep2mm
\item if $u,v\in \mathcal{PSH}(\Omega)$, then $su+tv\in \mathcal{PSH}(\Omega)$ for all constants $s,t\geq 0$;

\item if $u,v\in \mathcal{PSH}(\Omega)$, then $\max\{u,v\}\in \mathcal{PSH}(\Omega)$;

\item if $\{u_{\alpha}\}\subset \mathcal{PSH}(\Omega)$ is locally uniformly bounded above, then the upper semicontinuous regularization
        \[
        \left(\sup_{\alpha}u_{\alpha}\right)^*
        \]
        is plurisubharmonic on $\Omega$;

\item if $\{u_j\}$ is a sequence in $\mathcal{PSH}(\Omega)$ such that $u_j \searrow u$ and $u$ is not identically $-\infty$ on any component of $\Omega$, then $u \in \mathcal{PSH}(\Omega)$;

\item Let $u:\Omega\to[-\infty,\infty)$ be upper semicontinuous and
not identically $-\infty$ on any component of $\Omega$. Then $u$ is
plurisubharmonic if and only if, on each component of $\Omega\cap E$,
its restriction is plurisubharmonic or identically $-\infty$, for every
finite-dimensional complex subspace $E\subset H$.

\item if $\Omega$ is a domain and $u\in \mathcal{PSH}(\Omega)$ satisfies $u(z)\leq u(a)$ for some $a\in \Omega$ and all $z\in \Omega$, then $u\equiv u(a)$ (maximum principle);

\item if $u\in \mathcal{PSH}(\Omega)$ and $\gamma:\mathbb R\to\mathbb R$ is convex and nondecreasing, then $\gamma\circ u\in\mathcal{PSH}(\Omega)$. Here $\gamma(-\infty):=\lim_{t\to-\infty}\gamma(t)$.

\item if $\omega$ is a bounded open set with $\overline\omega\subset\Omega$, $u\in \mathcal{PSH}(\Omega)$, $v\in \mathcal{PSH}(\omega)$, and $\varlimsup_{z\to w}v(z)\leq u(w)$ for all $w\in \partial \omega$, then the function
\[
\varphi=\begin{cases}
u, \, \text{on } \Omega\setminus \omega ,\\
\max\{u,v\}, \, \text{on } \omega,
\end{cases}
\]
is plurisubharmonic on $\Omega$;

\item if $u$ is $C^2$-smooth, then $u\in\mathcal{PSH}(\Omega)$ if and only if the Levi form $D'D''u(z)$ is positive semidefinite, i.e.\ $\langle L_u(z)h,h\rangle\geq 0$ for all $z\in \Omega$ and $h\in H$.
\end{enumerate}

\end{theorem}

\begin{remark}
For a positive operator $L_u(z)$, the condition
\[
\langle L_u(z)h,h\rangle>0\qquad(h\ne0)
\]
is equivalent to injectivity. In infinite dimensions this need not
give a lower bound $L_u(z)\ge mI$ with $m>0$.
Consequently, a maximal function may have an injective Levi
operator at every point. We state these properties directly to
distinguish injectivity from positive invertibility.
\end{remark}

\begin{lemma}\label{lem:boundary_max}
Let $G$ be a bounded open subset of a nonzero complex Hilbert space
and let $w\in\mathcal{PSH}(G)$. If
\[
\limsup_{G\ni z\to\xi}w(z)\le0\qquad(\xi\in\partial G),
\]
then $w\le0$ in $G$.
\end{lemma}
\begin{proof}
Fix $a\in G$ and a unit vector $x$. Let $U$ be the connected
component containing $0$ of
$\{t\in\mathbb C\colon a+tx\in G\}$. Then $U$ is a bounded
planar domain. The function $t\mapsto w(a+tx)$ is subharmonic on
$U$ or identically $-\infty$. The latter case gives $w(a)\le0$
immediately. If $\tau\in\partial U$, then $a+\tau x\in\partial G$.
Indeed, if $a+\tau x\in G$, a small disc about $\tau$ would lie
in the open slice and meet $U$, contradicting the definition of
the component. The boundary hypothesis therefore gives
\[
\limsup_{U\ni t\to\tau}w(a+tx)\le0
\qquad(\tau\in\partial U).
\]
The planar maximum principle yields $w(a)\le0$.
\end{proof}

\section{Maximal plurisubharmonic functions}

In this section, we recall the notion of maximal plurisubharmonic functions and relate it to the Fujii--Seo determinant density.

\begin{definition}
A function $u\in \mathcal{PSH}(\Omega)$ is called \emph{maximal} in $\Omega$ if for every bounded open set $G$ with $\overline G\subset\Omega$ and every $v\in \mathcal{PSH}(G)$ such that
\[
\varlimsup_{z\to \xi}v(z)\leq u(\xi)\qquad\text{for all }\xi\in \partial G,
\]
one has $v\leq u$ on $G$.
\end{definition}

We call $u$ locally maximal if every point of $\Omega$
has an open neighborhood $U\subset\Omega$ on which $u$ is maximal.

\begin{definition}
Let $\Omega\subset H$ be open and let $u\in \mathcal{PSH}(\Omega)\cap C^2(\Omega)$. We define the \emph{Fujii--Seo determinant density} of $u$ by
\[
\operatorname{FSD}(u)(a):=\inf_{\|x\|=1}\Delta_x\!\bigl(L_u(a)\bigr),
\qquad a\in \Omega.
\]
\end{definition}

We first prove that maximality forces the determinant
density to vanish. We then prove the converse under an explicit
boundedness assumption and obtain a characterization of local
maximality.

\begin{theorem}\label{maximality_implies_zero}
Let $\Omega\subset H$ be open and let $u\in \mathcal{PSH}(\Omega)\cap C^2(\Omega)$. If $u$ is maximal in $\Omega$, then
\[
\operatorname{FSD}(u)(a)=0 \qquad\text{for every }a\in \Omega.
\]
\end{theorem}

\begin{proof}
Assume that $\operatorname{FSD}(u)(a)>0$ at some point $a\in \Omega$. Pick $\delta>0$ such that $\operatorname{FSD}(u)(a)>2\delta$. By Proposition~\ref{prop:extrema} applied to the positive operator $L_u(a)$, we have
\[
\operatorname{FSD}(u)(a)=\inf\sigma(L_u(a)).
\]
Hence $L_u(a)\ge 2\delta I$. By continuity of $L_u$ there exists $r>0$ with $\overline{B(a,r)}\subset\Omega$ and
\[
L_u(z)\ge \delta I\qquad\text{for all }z\in B(a,r).
\]
Define
\[
v(z):=u(z)-\frac\delta2\|z-a\|^2+\frac\delta2 r^2,
\qquad z\in B(a,r).
\]
Since $L_{\|\cdot-a\|^2}=I$, we have
\[
L_v(z)=L_u(z)-\frac\delta2 I\ge \frac\delta2 I\ge 0,
\]
so $v\in \mathcal{PSH}(B(a,r))$. On the boundary $\partial B(a,r)$ we have $\|z-a\|=r$ and hence $v(z)=u(z)$. However,
\[
v(a)=u(a)+\frac\delta2 r^2>u(a),
\]
contradicting the maximality of $u$. Therefore $\operatorname{FSD}(u)(a)=0$ for all $a\in\Omega$.
\end{proof}

The next theorem gives the converse under a boundedness
assumption on the function.

\begin{theorem}\label{thm:bounded_converse}
Let $\Omega\subset H$ be open and let
$u\in\mathcal{PSH}(\Omega)\cap C^2(\Omega)$. Assume that $u$ is
bounded above and below on $\overline G$ for every bounded open set
$G$ with $\overline G\subset\Omega$. Then $u$ is maximal in
$\Omega$ if and only if
\[
\inf\sigma\bigl(L_u(a)\bigr)=0
\qquad(a\in\Omega).
\]
\end{theorem}

\begin{proof}
Necessity follows from Theorem~\ref{maximality_implies_zero} and
Proposition~\ref{prop:extrema}. For sufficiency, suppose that the
spectral condition holds. Let $G$ be a bounded open set with
$\overline G\subset\Omega$, and let $v\in\mathcal{PSH}(G)$ satisfy
\[
\varlimsup_{G\ni z\to\xi}v(z)\le u(\xi)
\qquad(\xi\in\partial G).
\]
Suppose, contrary to the desired comparison, that
\[
a:=v(z_0)-u(z_0)>0
\]
at some $z_0\in G$. Put
\[
m:=\inf_{\overline G}u> -\infty,
\qquad M:=\sup_{\overline G}u<\infty,
\qquad R^2:=\sup_{\overline G}\|z\|^2<\infty.
\]
The boundary inequality and Lemma~\ref{lem:boundary_max} give $v\le M$ on $G$.

Choose $\delta>0$ so that $\delta R^2<a/4$.

On the underlying real Hilbert space, define
\[
f(z):=
\begin{cases}
u(z)-v(z)-\delta\|z\|^2,&z\in G,\\
-\delta\|z\|^2,&z\in\partial G,\\
+\infty,&z\notin\overline G.
\end{cases}
\]
Here $f(z)=+\infty$ at points of $G$ where $v(z)=-\infty$.
The function $f$ is lower semicontinuous on $G$. At each
$\xi\in\partial G$, continuity of $u$ and the boundary condition
give
\[
\liminf_{G\ni z\to\xi}f(z)
\ge -\delta\|\xi\|^2=f(\xi),
\]
so the displayed extension is lower semicontinuous on $H$.
Furthermore, $f\ge m-M-\delta R^2$ on $G$ and
$f\ge-\delta R^2$ on $\partial G$. Thus $f$ is bounded below.
It is finite at $z_0$, and
\[
\inf_H f\le f(z_0)\le-a.
\]

Choose $0<\varepsilon<\min\{\delta,a/4\}$ and $x_0\in H$ such that
\[
f(x_0)<\inf_H f+\varepsilon.
\]
Apply the smooth variational principle of Borwein and Preiss
\cite[Theorem~2.6]{BP87} with exponent $2$ and distance parameter $1$.
By the definition of the perturbation class in \cite[(2.1)]{BP87},
there are nonnegative numbers $\alpha_i$ with
$\sum_{i=1}^{\infty}\alpha_i=1$, a norm-convergent sequence
$(x_i)\subset H$, and $\bar z\in H$ such that $f+q$ attains its
global minimum at $\bar z$, where
\[
q(z):=\varepsilon\sum_{i=1}^{\infty}
\alpha_i\|z-x_i\|^2.
\]
Conclusion (d) of that theorem also gives
\[
f(\bar z)<\inf_H f+\varepsilon<-3a/4.
\]
Since $f>-a/4$ on $\partial G$ and $f=+\infty$ outside
$\overline G$, it follows that $\bar z\in G$ and
$v(\bar z)>-\infty$.

The centers $x_i$ are bounded, so the series
$b:=\sum_i\alpha_i x_i$ and
$c:=\sum_i\alpha_i\|x_i\|^2$ converge. Consequently
\[
q(z)=\varepsilon\|z\|^2
-2\varepsilon\operatorname{Re}\langle z,b\rangle+\varepsilon c,
\qquad L_q=\varepsilon I.
\]
Let $C:=f(\bar z)+q(\bar z)$ and define
\[
\phi(z):=u(z)-\delta\|z\|^2+q(z)-C
\qquad(z\in G).
\]
The minimum property gives $v\le\phi$ on $G$, with equality at
$\bar z$. For every $h\in H$ and every sufficiently small $r>0$,
the submean inequality therefore gives
\[
\phi(\bar z)=v(\bar z)
\le\frac1{2\pi}\int_0^{2\pi}v(\bar z+re^{i\theta}h)\,d\theta
\le\frac1{2\pi}\int_0^{2\pi}\phi(\bar z+re^{i\theta}h)\,d\theta.
\]
Since $\phi$ is $C^2$, expansion at $r=0$ yields
$\langle L_\phi(\bar z)h,h\rangle\ge0$.
It follows that
\[
L_u(\bar z)\ge(\delta-\varepsilon)I,
\]
contrary to the spectral condition. This proves maximality.
\end{proof}

\begin{corollary}\label{cor:local_maximality}
Let $\Omega\subset H$ be open and let
$u\in\mathcal{PSH}(\Omega)\cap C^2(\Omega)$. Then $u$ is locally
maximal if and only if
\[
\inf\sigma\bigl(L_u(a)\bigr)=0
\qquad(a\in\Omega).
\]
\end{corollary}

\begin{proof}
If $u$ is locally maximal, Theorem~\ref{maximality_implies_zero}
applied to each such neighborhood gives the spectral condition.
Conversely, suppose that the spectral condition holds and fix
$a\in\Omega$. By continuity, there is $r>0$ such that
$\overline{B(a,r)}\subset\Omega$ and
$|u(z)-u(a)|<1$ for $z\in\overline{B(a,r)}$.
Thus $u$ is bounded above and below on this ball.
Theorem~\ref{thm:bounded_converse} applied on $B(a,r)$ shows that
$u$ is maximal there.
\end{proof}

The next theorem gives a pointwise condition for maximality without
boundedness assumptions on the function. In the proof, applying
$\arctan$ to the difference $v-u$, where $v$ is the comparison function,
allows us to use the smooth variational principle of Borwein and Preiss.
The additional term produced by the chain rule vanishes on the kernel
of a complex-linear functional.

\begin{theorem}\label{thm:hyperplane_maximality}
Let $H$ be a complex Hilbert space with $\dim_{\mathbb C}H\ge2$, let
$\Omega\subset H$ be open, and let
$u\in\mathcal{PSH}(\Omega)\cap C^2(\Omega)$. Assume that for every
$a\in\Omega$ and every continuous complex-linear functional $\ell$ on
$H$,
\begin{equation}\label{eq:hyperplane_null}
\inf_{\substack{\|h\|=1\\\ell(h)=0}}
\langle L_u(a)h,h\rangle=0.
\end{equation}
Then $u$ is maximal in $\Omega$.
\end{theorem}

\begin{proof}
Let $G$ be a bounded open set with $\overline G\subset\Omega$, and let
$v\in\mathcal{PSH}(G)$ satisfy
\[
\varlimsup_{G\ni z\to\xi}v(z)\le u(\xi)
\qquad(\xi\in\partial G).
\]
Suppose that $v(z_0)>u(z_0)$ for some $z_0\in G$. Define
\[
w(z):=\arctan\bigl(v(z)-u(z)\bigr),
\qquad \arctan(-\infty):=-\frac\pi2.
\]
The function $w$ is upper semicontinuous, takes values in
$[-\pi/2,\pi/2)$, and has nonpositive upper boundary limits. Put
\[
a:=w(z_0)>0,
\qquad R^2:=\sup_{z\in\overline G}\|z\|^2<\infty,
\]
and choose $\delta>0$ such that $\delta R^2<a/4$. On the underlying
real Hilbert space, define
\[
F(z):=
\begin{cases}
-w(z)-\delta\|z\|^2,&z\in G,\\
-\delta\|z\|^2,&z\in\partial G,\\
+\infty,&z\notin\overline G.
\end{cases}
\]
This function is lower semicontinuous in $G$. At every
$\xi\in\partial G$, the boundary inequality gives
\[
\liminf_{G\ni z\to\xi}F(z)
\ge-\delta\|\xi\|^2=F(\xi).
\]
The boundary formula is continuous, so $F$ is lower semicontinuous
on $H$. Moreover,
\[
F\ge-\frac\pi2-\delta R^2,
\qquad \inf_H F\le F(z_0)\le-a.
\]
Thus $F$ is bounded below and finite at $z_0$.

Choose $0<\varepsilon<\min\{\delta,a/4\}$ and $x_0\in H$ such that
\[
F(x_0)<\inf_H F+\varepsilon.
\]
Apply Borwein--Preiss~\cite[Theorem~2.6]{BP87} with exponent $2$ and
distance parameter $1$. By the definition of the perturbation class in
\cite[(2.1)]{BP87}, there are nonnegative numbers $\alpha_i$ with
$\sum_i\alpha_i=1$, a norm-convergent sequence $(x_i)$, and
$\bar z\in H$ such that $F+q$ attains its global minimum at $\bar z$,
where
\[
q(z):=\varepsilon\sum_{i=1}^{\infty}\alpha_i\|z-x_i\|^2.
\]
Conclusion (d) of that theorem gives
\[
F(\bar z)<\inf_H F+\varepsilon<-3a/4.
\]
Since $F>-a/4$ on $\partial G$ and $F=+\infty$ outside
$\overline G$, we have $\bar z\in G$. Furthermore,
\[
w(\bar z)=-F(\bar z)-\delta\|\bar z\|^2>a/2>0.
\]
In particular, $v(\bar z)$ is finite and $w(\bar z)\in(0,\pi/2)$.

The centers $x_i$ are bounded. Hence the series
$b:=\sum_i\alpha_i x_i$ and $c:=\sum_i\alpha_i\|x_i\|^2$ converge,
and
\[
q(z)=\varepsilon\|z\|^2
-2\varepsilon\operatorname{Re}\langle z,b\rangle+\varepsilon c,
\qquad L_q=\varepsilon I.
\]
Set $C:=F(\bar z)+q(\bar z)$ and
\[
s(z):=q(z)-\delta\|z\|^2-C.
\]
The minimum property gives $w\le s$ on $G$, with equality at
$\bar z$. On a neighborhood of $\bar z$, the values of $s$ lie in
$(-\pi/2,\pi/2)$. At points where $v=-\infty$, the comparison
below is automatic. At the remaining points, applying the increasing
function $\tan$ yields
\[
v(z)\le\phi(z):=u(z)+\tan s(z),
\qquad v(\bar z)=\phi(\bar z).
\]
For every $h\in H$ and every sufficiently small $r>0$, the submean
inequality gives
\[
\phi(\bar z)=v(\bar z)
\le\frac1{2\pi}\int_0^{2\pi}v(\bar z+re^{i\theta}h)\,d\theta
\le\frac1{2\pi}\int_0^{2\pi}\phi(\bar z+re^{i\theta}h)\,d\theta.
\]
Since $\phi$ is $C^2$, expansion at $r=0$ shows that
$L_\phi(\bar z)\ge0$.

Put $s_0:=s(\bar z)$ and $\ell:=D's(\bar z)$, which is a continuous
complex-linear functional. The chain rule gives
\begin{align}\label{eq:bounded_transform_contact}
\langle L_\phi(\bar z)h,h\rangle
={}&\langle L_u(\bar z)h,h\rangle
+(\varepsilon-\delta)\sec^2(s_0)\|h\|^2\notag\\
&+2\sec^2(s_0)\tan(s_0)|\ell(h)|^2.
\end{align}
By \eqref{eq:hyperplane_null}, there are unit vectors $h_n\in\ker\ell$
such that $\langle L_u(\bar z)h_n,h_n\rangle\to0$. Substituting
$h_n$ in \eqref{eq:bounded_transform_contact} and passing to the limit
gives
\[
0\le(\varepsilon-\delta)\sec^2(s_0)<0,
\]
a contradiction. Therefore $v\le u$ on $G$, and $u$ is maximal.
\end{proof}

For a bounded self-adjoint operator $A$, let $\sigma_{\mathrm{ess}}(A)$
denote its spectrum with the isolated eigenvalues of finite multiplicity
removed. The next corollary gives three consequences of
Theorem~\ref{thm:hyperplane_maximality}.

\begin{corollary}\label{cor:pointwise_levi_maximality}\label{prop:bounded_levi_rank}
Let $H$ be a complex Hilbert space with $\dim_{\mathbb C}H\ge2$, let
$\Omega\subset H$ be open, and let
$u\in\mathcal{PSH}(\Omega)\cap C^2(\Omega)$. Each of the following
conditions implies that $u$ is maximal in $\Omega$.
\begin{enumerate}
\item $H$ is infinite-dimensional and $L_u(a)$ is compact for every
$a\in\Omega$.
\item $0\in\sigma_{\mathrm{ess}}(L_u(a))$ for every $a\in\Omega$.
\item $\dim_{\mathbb C}\ker L_u(a)\ge2$ for every $a\in\Omega$.
\end{enumerate}
If $H$ is infinite-dimensional, then $u$ is maximal if $L_u(a)$ has
finite rank at every $a\in\Omega$. No bound on these ranks, uniform
in $a$, is required.
\end{corollary}

\begin{proof}
Suppose first that $0\in\sigma_{\mathrm{ess}}(A)$, where
$A:=L_u(a)$. By the spectral theorem, the spectral subspace
$E_A([0,t))H$ is infinite-dimensional for every $t>0$. Given a
continuous complex-linear functional $\ell$, choose a unit vector
$h_t\in E_A([0,t))H\cap\ker\ell$. Then
\[
0\le\langle Ah_t,h_t\rangle\le t.
\]
Letting $t\searrow0$ proves \eqref{eq:hyperplane_null}. A compact
positive operator on an infinite-dimensional Hilbert space has zero in
its essential spectrum, so the first condition is included in the
second. Under the third condition, $\ker A\cap\ker\ell$ contains a
unit vector, which again gives \eqref{eq:hyperplane_null}.
The conclusion follows from Theorem~\ref{thm:hyperplane_maximality}.
Every finite-rank operator is compact, so the last assertion follows
from the first condition.
\end{proof}

For a bounded positive operator $A$ on an infinite-dimensional Hilbert
space, the condition
\[
\inf_{\substack{\|h\|=1\\\ell(h)=0}}\langle Ah,h\rangle=0
\qquad\text{for every continuous complex-linear functional }\ell
\]
is equivalent to $0\in\sigma_{\mathrm{ess}}(A)$ or
$\dim_{\mathbb C}\ker A\ge2$. The preceding proof gives sufficiency.
Conversely, taking $\ell=0$ gives $\inf\sigma(A)=0$. If zero is not
in the essential spectrum, it is an isolated eigenvalue of finite
multiplicity. If that multiplicity were one, with unit eigenvector
$e$, then $A$ would be bounded below by a positive constant on
$e^\perp$. Taking $\ell(h)=\langle h,e\rangle$ would contradict the
displayed condition.

Whether local maximality implies maximality without the boundedness
assumption in Theorem~\ref{thm:bounded_converse} remains open.
The preceding theorem and corollary give sufficient conditions without
this assumption. We next consider conditions involving directions that
can be chosen uniformly on each bounded comparison set.

We compare the following conditions on the Levi operators and their Fujii--Seo determinants:
\begin{itemize}\itemsep2mm
\item[$(1)$] for every $a\in \Omega$ there exists a unit vector $x_a$ such that $\Delta_{x_a}(L_u(a))=0$;
\item[$(2)$] there exists a unit vector $x$ such that $\Delta_x(L_u(a))=0$ for all $a\in\Omega$;
\item[$(3)$] for every $a\in \Omega$ there exists a unit vector $x_a$ such that $\langle L_u(a)x_a,x_a\rangle=0$;
\item[$(4)$] there exists a unit vector $x$ such that $\langle L_u(a)x,x\rangle=0$ for all $a\in\Omega$.
\end{itemize}

The remark following Proposition~\ref{bp2} shows that
condition $(1)$ is equivalent to $\operatorname{FSD}(u)\equiv0$.
The existence of a vector $x_a$ at each point does not imply the
condition in Theorem~\ref{prop:approx_null_maximality}.

Conditions $(1)$--$(2)$ do not imply $(3)$--$(4)$ in general: Examples~\ref{e111} and~\ref{eL2} below show that $(2)$ may hold even when $L_u(a)$ is injective for every $a$.

The following theorem gives a uniform sufficient condition for
maximality without any boundedness assumption on $u$.

\begin{theorem}\label{prop:approx_null_maximality}
Let $\Omega\subset H$ be open and let $u\in \mathcal{PSH}(\Omega)\cap C^2(\Omega)$. Assume that for every bounded open set $G$ with $\overline G\subset\Omega$ there exists a sequence $(x_n)$ of unit vectors in $H$ such that
\[
\sup_{z\in G}\langle L_u(z)x_n,x_n\rangle \longrightarrow 0
\qquad\text{as }n\to\infty.
\]
Then $u$ is maximal in $\Omega$.
\end{theorem}
\begin{proof}
Assume to the contrary that $u$ is not maximal. Then there exist a bounded open set $G$ with $\overline G\subset\Omega$, a function $v\in \mathcal{PSH}(G)$, and a point $z_0\in G$ such that
\[
\varlimsup_{z\to \xi}v(z)\leq u(\xi)
\qquad\text{for all }\xi\in \partial G,
\]
but
\[
v(z_0)>u(z_0).
\]
Set
\[
R:=\sup_{z\in \overline G}\|z-z_0\|<\infty.
\]
By assumption, there is a sequence of unit vectors $(x_n)$ such that
\[
\lambda_n:=\sup_{z\in G}\langle L_u(z)x_n,x_n\rangle \longrightarrow 0.
\]
After omitting finitely many terms, we may assume that
$\lambda_n<\infty$ for every $n$.
For each $n$, let $U_n$ be the connected component containing $0$ of the open set
\[
\{t\in \mathbb C:\ z_0+t x_n\in G\}.
\]
If $t\in U_n$, then $z_0+t x_n\in G$, hence
\[
|t|=\|t x_n\|=\|z_0+t x_n-z_0\|\leq R.
\]
Define
\[
w_n(t):=
v(z_0+t x_n)-u(z_0+t x_n)+\lambda_n\bigl(|t|^2-R^2\bigr),
\qquad t\in U_n.
\]
Since $v$ is plurisubharmonic, the function $t\mapsto v(z_0+t x_n)$ is subharmonic on $U_n$. Set
\[
u_n(t):=u(z_0+t x_n), \qquad t\in U_n.
\]
Since $u\in C^2(\Omega)$, writing $t=s+ir$ and using the formula for the Levi form gives
\[
\Delta_t u_n(t)
=
\frac{\partial^2 u_n}{\partial s^2}(t)
+
\frac{\partial^2 u_n}{\partial r^2}(t)
=
4\langle L_u(z_0+t x_n)x_n,x_n\rangle
\leq 4\lambda_n,
\]
where $\Delta_t=\partial_s^2+\partial_r^2$. Hence
\[
\Delta_t\bigl(-u_n(t)+\lambda_n |t|^2\bigr)
=
-\Delta_t u_n(t)+4\lambda_n
\geq 0.
\]
Thus the function
\[
t\longmapsto -u(z_0+t x_n)+\lambda_n|t|^2
\]
is subharmonic on $U_n$. Therefore, $w_n$ is subharmonic on $U_n$.

Now let $\tau\in \partial U_n$. We first show that $z_0+\tau x_n\in \partial G$.
Indeed, choose $t_k\in U_n$ with $t_k\to\tau$. Then $z_0+t_kx_n\in G$, hence $z_0+\tau x_n\in \overline G$.
If $z_0+\tau x_n\in G$, then $\tau$ belongs to the open slice $\{t\in\mathbb C:\ z_0+t x_n\in G\}$.
Since this slice is open and $\tau\in\partial U_n$, a small disc around $\tau$ contained in the slice intersects $U_n$; the union of this disc with $U_n$ is then a connected subset of the slice containing $0$. By maximality of the connected component $U_n$, this would force $\tau\in U_n$, a contradiction.
Thus $z_0+\tau x_n\in \partial G$. Since also $|\tau|\leq R$, we obtain
\[
\varlimsup_{t\to \tau} w_n(t)\leq 0.
\]
By the maximum principle for subharmonic functions on planar domains,
\[
w_n(0)\leq 0.
\]
Thus
\[
v(z_0)-u(z_0)-\lambda_n R^2\leq 0,
\]
that is,
\[
v(z_0)-u(z_0)\leq \lambda_n R^2.
\]
Letting $n\to\infty$ gives $v(z_0)\leq u(z_0)$, a contradiction.
\end{proof}

\begin{remark}
The hypothesis of Theorem~\ref{prop:approx_null_maximality} is
sufficient but not necessary for maximality. For the maximal function
in Example~\ref{e:rank_one_no_uniform_null}, the Levi rank is at most
one, but
\[
\inf_{\|h\|=1}\sup_{\|z\|<R}
\langle L_u(z)h,h\rangle=4R^2
\qquad(R>0).
\]
\end{remark}

The next corollary applies when the Levi operators have a common
nonzero vector in their kernels. Since they are self-adjoint, this is
equivalent to requiring all their ranges to lie in a fixed proper
closed subspace of $H$.

\begin{corollary}\label{maximal}\label{prop:common-range}
Let $\Omega\subset H$ be open and let
$u\in\mathcal{PSH}(\Omega)\cap C^2(\Omega)$. If there exists a unit
vector $x\in H$ such that
\[
\langle L_u(a)x,x\rangle=0\qquad(a\in\Omega),
\]
then $u$ is maximal in $\Omega$. Equivalently, it suffices that there
exist a proper closed subspace $E\subsetneq H$ such that
\[
\operatorname{Ran}(L_u(a))\subset E\qquad(a\in\Omega).
\]
\end{corollary}

\begin{proof}
The first assertion follows from
Theorem~\ref{prop:approx_null_maximality} with $x_n=x$.
For a positive operator $A$, the equality $\langle Ax,x\rangle=0$
is equivalent to $Ax=0$. If $L_u(a)x=0$ for all $a$,
self-adjointness gives
$\operatorname{Ran}(L_u(a))\subset x^\perp$ for all $a$.
Conversely, if the ranges lie in $E$, every unit vector $x\in E^\perp$
satisfies $\langle L_u(a)x,x\rangle=0$ for all $a$.
\end{proof}

\begin{corollary}\label{cor:constant_levi}
Let $\Omega$ be a nonempty open subset of $H$ and let
$u\in\mathcal{PSH}(\Omega)\cap C^2(\Omega)$. Assume that there exists a bounded positive operator $A\in \mathcal B(H)$ such that
\[
L_u(a)=A
\qquad\text{for all }a\in \Omega.
\]
Then the following are equivalent:
\begin{enumerate}\itemsep2mm
\item $u$ is maximal in $\Omega$;
\item $\operatorname{FSD}(u)\equiv 0$ on $\Omega$;
\item $\inf\sigma(A)=0$.
\end{enumerate}
\end{corollary}

\begin{proof}
$(1)\Rightarrow(2)$ is Theorem~\ref{maximality_implies_zero}.

Since $L_u(a)=A$ for every $a\in \Omega$, Proposition~\ref{prop:extrema} gives
\[
\operatorname{FSD}(u)(a)
=
\inf_{\|x\|=1}\Delta_x(A)
=
\inf\sigma(A)
\qquad\text{for all }a\in \Omega.
\]
Hence $(2)\Leftrightarrow(3)$.

Assume now that $\inf\sigma(A)=0$. By Proposition~\ref{bp2}, there exists a sequence $(x_n)$ of unit vectors such that
\[
\langle Ax_n,x_n\rangle\to 0.
\]
Therefore, for every bounded open set $G$ with $\overline G\subset\Omega$,
\[
\sup_{z\in G}\langle L_u(z)x_n,x_n\rangle
=
\langle Ax_n,x_n\rangle
\longrightarrow 0.
\]
Theorem~\ref{prop:approx_null_maximality} now implies that $u$ is maximal.
\end{proof}

For a closed subspace $E\subset H$, let $P_E$ denote the orthogonal
projection onto $E$.

\begin{proposition}\label{prop:approx_common_range}
Let $\Omega\subset H$ be open, and let $u\in \mathcal{PSH}(\Omega)\cap C^2(\Omega)$. Assume that for every bounded open set $G$ with $\overline G\subset\Omega$ and every $\varepsilon>0$ there exists a proper closed subspace
\[
E_{\varepsilon,G}\subsetneq H
\]
such that
\[
\|P_{E_{\varepsilon,G}^\perp}L_u(z)\|\le \varepsilon
\qquad (z\in G).
\]
Then $u$ is maximal in $\Omega$.
\end{proposition}

\begin{proof}
Fix a bounded open set $G$ with $\overline G\subset\Omega$. For each $n\in\N$, choose a proper closed subspace
\[
E_n:=E_{1/n,G}\subsetneq H
\]
such that
\[
\|P_{E_n^\perp}L_u(z)\|\le \frac1n
\qquad (z\in G).
\]
Choose a unit vector
\[
x_n\in E_n^\perp.
\]
Then for every $z\in G$,
\[
\langle L_u(z)x_n,x_n\rangle
=
\langle P_{E_n^\perp}L_u(z)x_n,x_n\rangle
\le
\|P_{E_n^\perp}L_u(z)\|
\le
\frac1n.
\]
Hence
\[
\sup_{z\in G}\langle L_u(z)x_n,x_n\rangle\longrightarrow 0.
\]
By Theorem~\ref{prop:approx_null_maximality}, $u$ is maximal in $\Omega$.
\end{proof}

\begin{remark}
Assume that $H$ is infinite-dimensional. A family $\mathcal A$ of
bounded operators on $H$ is called collectively compact if
\[
\{Ah:A\in\mathcal A,\ \|h\|\le1\}
\]
is relatively compact in $H$. If
$\{L_u(z):z\in G\}$ is collectively compact for every bounded
open set $G$ with $\overline G\subset\Omega$, then the hypothesis of
Proposition~\ref{prop:approx_common_range} holds. Indeed, for
$\varepsilon>0$, choose a finite $\varepsilon$-net for the displayed
set and let $E$ be its linear span. Then $E$ is finite-dimensional and
proper, and
\[
\|P_{E^\perp}L_u(z)\|\le\varepsilon\qquad(z\in G).
\]
This proves maximality by the uniform criterion. Maximality itself
already follows from Corollary~\ref{cor:pointwise_levi_maximality},
since each operator in a collectively compact family is compact.
Example~\ref{e:rank_one_no_uniform_null} shows that compactness of
each Levi operator does not imply either collective compactness or the
hypothesis of Theorem~\ref{prop:approx_null_maximality}.
\end{remark}

The following proposition gives another sufficient condition for maximality. It is useful when the Levi operator is dominated on bounded open sets by a fixed positive operator whose spectrum has an infimum equal to $0$.

\begin{proposition}\label{prop:model_majorant}
Let $\Omega\subset H$ be open and let $u\in \mathcal{PSH}(\Omega)\cap C^2(\Omega)$. Assume that for every bounded open set $G$ with $\overline G\subset\Omega$, there exists a bounded positive operator
\[
T_G\in \mathcal B(H)
\]
such that
\[
L_u(z)\le T_G
\qquad (z\in G),
\]
and
\[
\inf\sigma(T_G)=0.
\]
Then $u$ is maximal in $\Omega$.
\end{proposition}

\begin{proof}
Fix a bounded open set $G$ with $\overline G\subset\Omega$. By Proposition~\ref{bp2}, there exists a sequence $(x_n)$ of unit vectors such that
\[
\langle T_Gx_n,x_n\rangle\to 0.
\]
Then, for every $z\in G$,
\[
0\le \langle L_u(z)x_n,x_n\rangle\le \langle T_Gx_n,x_n\rangle.
\]
Hence
\[
\sup_{z\in G}\langle L_u(z)x_n,x_n\rangle\longrightarrow 0.
\]
By Theorem~\ref{prop:approx_null_maximality}, $u$ is maximal in $\Omega$.
\end{proof}

The following proposition gives a class of examples used in Section~5.

\begin{proposition}\label{prop:phiA_factory}
Let $A\in \mathcal B(H)$ be a positive operator, and let $\Phi\in C^2([0,\infty))$ satisfy
\[
\Phi'(t)+t\Phi''(t)\ge 0
\qquad (t\ge 0).
\]
Define
\[
u(z):=\Phi(\langle Az,z\rangle),
\qquad z\in H.
\]
Then
\[
u\in \mathcal{PSH}(H)\cap C^2(H),
\]
and
\[
L_u(z)
=
\Phi'(\langle Az,z\rangle)\,A
+
\Phi''(\langle Az,z\rangle)\,R_z,
\]
where
\[
R_z h:=\langle h,Az\rangle\,Az.
\]
If, in addition,
\[
\inf\sigma(A)=0,
\]
then $u$ is maximal on $H$. In particular,
\[
\operatorname{FSD}(u)\equiv 0.
\]
If $\Phi\in C^\infty([0,\infty))$, then $u\in C^\infty(H)$.

\end{proposition}

\begin{proof}
Put $q(z):=\langle Az,z\rangle$. The function $t\mapsto t\Phi'(t)$
is nondecreasing and vanishes at zero, so $\Phi'(t)\ge0$ for $t\ge0$.
The chain rule gives
\[
\langle L_u(z)h,h\rangle
=\Phi'(q(z))\langle Ah,h\rangle
 +\Phi''(q(z))|\langle Az,h\rangle|^2.
\]
The Cauchy--Schwarz inequality applied to $A^{1/2}z$ and $A^{1/2}h$
gives
\[
|\langle Az,h\rangle|^2\le q(z)\langle Ah,h\rangle,
\qquad 0\le R_z\le q(z)A.
\]
If $\Phi''(q(z))\ge0$, positivity of the Levi operator follows from
$\Phi'(q(z))\ge0$. If $\Phi''(q(z))<0$, the same inequality gives
\[
\langle L_u(z)h,h\rangle
\ge\bigl(\Phi'(q(z))+q(z)\Phi''(q(z))\bigr)\langle Ah,h\rangle
\ge0.
\]
Thus $u$ is plurisubharmonic. Its asserted differentiability follows
from that of $\Phi$ and the continuous quadratic polynomial $q$.

Assume now that $\inf\sigma(A)=0$, and let $G\subset H$ be bounded
and open. Set
\[
M:=\|A\|\sup_{z\in G}\|z\|^2,
\qquad
C_G:=\sup_{0\le t\le M}\bigl(\Phi'(t)+t|\Phi''(t)|\bigr).
\]
Then $C_G<\infty$, and the inequalities above give
\[
0\le L_u(z)\le C_GA\qquad(z\in G).
\]
Since $\inf\sigma(C_GA)=0$, Proposition~\ref{prop:model_majorant}
proves maximality. Theorem~\ref{maximality_implies_zero} gives
$\operatorname{FSD}(u)\equiv0$.
\end{proof}

\begin{remark}
There is also a useful sufficient condition for diagonal Levi operators.
Suppose that $u\in\mathcal{PSH}(\ell^2)\cap C^2(\ell^2)$ and
\[
L_u(z)=\operatorname{diag}(b_j(z)),\qquad b_j(z)\ge0.
\]
If for every bounded open set $G\subset\ell^2$ there is a sequence of
nonempty finite sets $I_n\subset\mathbb N$ such that
\[
\sup_{z\in G}\frac1{|I_n|}\sum_{j\in I_n}b_j(z)\longrightarrow0,
\]
then $u$ is maximal. Indeed, the unit vectors
\[
x_n:=\frac1{\sqrt{|I_n|}}\sum_{j\in I_n}e_j
\]
satisfy
\[
\langle L_u(z)x_n,x_n\rangle
=\frac1{|I_n|}\sum_{j\in I_n}b_j(z),
\]
so Theorem~\ref{prop:approx_null_maximality} applies.
\end{remark}

A decreasing sequence of maximal plurisubharmonic functions has a
maximal limit whenever the limit is not identically $-\infty$ on any
component. Indeed, the decreasing limit is plurisubharmonic, and a
boundary comparison with the limit is also a boundary comparison with
every member of the sequence. This observation applies in both finite-
and infinite-dimensional Hilbert spaces.
Increasing limits need not preserve maximality in infinite dimensions,
even for a uniformly bounded sequence, as the following example shows.

Let $H=\ell^2$ and let $\Omega:=B(0,1)\subset H$ be the open unit ball. For $j\ge 1$ define
\[
u_j(z):=\sum_{k=1}^j |z_k|^2,
\qquad z=(z_k)_{k=1}^{\infty}\in \Omega.
\]
Then $0\le u_j\le 1$ on $\Omega$, and each $u_j$ is maximal in $\Omega$ by Corollary~\ref{maximal} (take $x=e_{j+1}$). Moreover,
\[
u_j(z)\nearrow u(z):=\|z\|^2 \qquad \text{for every } z\in \Omega,
\]
since
\[
u(z)-u_j(z)=\sum_{k=j+1}^{\infty}|z_k|^2 \longrightarrow 0 \qquad \text{as } j\to\infty.
\]
The limit $u$ is not maximal in $\Omega$. To see this,
fix $0<r<1$ and set $G:=B(0,r)$. Then $\overline G\subset\Omega$, and $u\equiv r^2$ on $\partial G$, while the constant function $v\equiv r^2$ is plurisubharmonic on $G$ and satisfies $v>u$ in $G$.
Thus, maximality is not stable under increasing limits in infinite dimensions.

The proof of Theorem~\ref{thm:hyperplane_maximality} also gives a
necessary condition for failure of comparison when the infimum of the
spectrum of $L_u(a)$ is zero for every $a\in\Omega$.

\begin{proposition}\label{prop:simple_zero_obstruction}
Let $H$ be an infinite-dimensional complex Hilbert space, let
$\Omega\subset H$ be open, and let
$u\in\mathcal{PSH}(\Omega)\cap C^2(\Omega)$ satisfy
\[
\inf\sigma\bigl(L_u(a)\bigr)=0
\qquad(a\in\Omega).
\]
Suppose that $G$ is a bounded open set with $\overline G\subset\Omega$
and that $v\in\mathcal{PSH}(G)$ satisfies
\[
\varlimsup_{G\ni z\to\xi}v(z)\le u(\xi)
\qquad(\xi\in\partial G),
\]
but $v(z_0)>u(z_0)$ at some $z_0\in G$. Then there is a point
$\bar z\in G$ at which zero is an isolated eigenvalue of
$L_u(\bar z)$ of multiplicity one.
\end{proposition}

\begin{proof}
The construction in the proof of
Theorem~\ref{thm:hyperplane_maximality} up to
\eqref{eq:bounded_transform_contact} does not use
\eqref{eq:hyperplane_null}. It gives a point $\bar z\in G$, a
continuous complex-linear functional $\ell$, and a number
$c=(\delta-\varepsilon)\sec^2(s_0)>0$ such that, with
$A:=L_u(\bar z)$,
\[
\langle Ah,h\rangle\ge c\|h\|^2
\qquad(h\in\ker\ell).
\]
Since $\inf\sigma(A)=0$, the functional $\ell$ is nonzero. Fix
$0<t<c$, and let $P:=E_A([0,t))$ be the corresponding spectral
projection. If $\dim PH\ge2$, there is a unit vector
$h\in PH\cap\ker\ell$. It would satisfy
\[
c\le\langle Ah,h\rangle\le t<c,
\]
which is impossible. Also $P\ne0$, since otherwise $A\ge tI$.
Thus $PH$ is one-dimensional. The restriction of $A$ to $PH$ is
multiplication by some $\lambda\in[0,t)$, while
$A|_{(PH)^\perp}\ge tI$. The equality $\inf\sigma(A)=0$ forces
$\lambda=0$. Hence zero is an isolated eigenvalue of multiplicity one.
\end{proof}

\section{Examples}

The examples below show that $\operatorname{FSD}(u)\equiv 0$ is compatible with injectivity of the Levi operator at every point. In particular, conditions $(1)$--$(2)$ above do not imply the kernel conditions $(3)$--$(4)$.

\begin{example}\label{e111}\label{e112}
Let $H=\ell^2$ and define
\[
u(z):=\sum_{j=1}^{\infty}\frac{|z_j|^2}{j},
\qquad z=(z_j)\in H.
\]
Then $u\in \mathcal{PSH}(H)\cap C^{\infty}(H)$ and
\[
L_u(z)=T:=\mathrm{diag}\!\left(\frac1j\right)
\qquad (z\in H).
\]
The operator $T$ is positive, compact, and injective, and $\inf\sigma(T)=0$. Hence
\[
\operatorname{FSD}(u)\equiv 0,
\]
and $u$ is maximal on $H$ by Corollary~\ref{cor:constant_levi}.

Moreover, condition $(2)$ above holds although condition $(4)$ fails. Indeed, let
\[
C^2:=\sum_{j=1}^{\infty}\frac{1}{j\log^2(j+1)}<\infty,
\qquad
x_j:=\frac{1}{C\sqrt{j}\log(j+1)}.
\]
Then $x=(x_j)\in \ell^2$ and $\|x\|=1$. Since
\[
\langle (\log T)x,x\rangle
=-\frac{1}{C^2}\sum_{j=1}^{\infty}\frac{\log j}{j\log^2(j+1)}=-\infty,
\]
we have
\[
\Delta_x(T)=0.
\]
Here $\langle(\log T)x,x\rangle$ denotes the extended spectral integral used in the definition of $\Delta_x$. Since $T$ is an injective diagonal operator, there is no nonzero $h\in H$ with $\langle Th,h\rangle=0$. Thus a fixed unit vector $x$ can satisfy $\Delta_x(L_u(a))=0$ for every $a\in H$, although each Levi operator has trivial kernel.
More generally, let $(w_j)$ be a bounded sequence of nonnegative
numbers and set
\[
u_w(z):=\sum_{j=1}^{\infty}w_j|z_j|^2.
\]
Then $u_w\in\mathcal{PSH}(\ell^2)\cap C^\infty(\ell^2)$ and
$L_{u_w}(z)=\operatorname{diag}(w_j)$. Consequently,
\[
\operatorname{FSD}(u_w)=\inf_jw_j,
\qquad
u_w\text{ is maximal on }\ell^2\quad\Longleftrightarrow\quad
\inf_jw_j=0
\]
by Corollary~\ref{cor:constant_levi}.
If every $w_j$ is positive and $w_j\to0$, its Levi operator is compact
and injective.

\hfill{$\Box$}
\end{example}

\begin{remark}
Let $A$ be a bounded positive operator with $\inf\sigma(A)=0$ and
put $q(z):=\langle Az,z\rangle$. For every bounded open set $G$
containing zero,
\[
\inf_{\xi\in\partial G}q(\xi)=0.
\]
Indeed, choose $0<r<R$ with $B(0,r)\subset G\subset B(0,R)$.
By Proposition~\ref{bp2}, there are unit vectors $x_n$ such that
$\langle Ax_n,x_n\rangle\to0$. Set
\[
t_n:=\sup\{t>0:t x_n\in G\}.
\]
Then $r\le t_n\le R$ and $t_nx_n\in\partial G$, so
\[
0\le q(t_nx_n)
\le R^2\langle Ax_n,x_n\rangle\longrightarrow0.
\]
If $A$ is injective, then $q$ is positive at every boundary point,
although its boundary infimum is zero. Consequently, the inequality
$c\le q$ on $\partial G$ fails for every $c>0$.
\end{remark}

\begin{example}\label{eL2}
Let $H=L^2([0,1],dt)$ and define
\[
u(h):=\int_0^1 t\,|h(t)|^2\,dt,
\qquad h\in H.
\]
Then $u\in \mathcal{PSH}(H)\cap C^{\infty}(H)$ and
\[
L_u(h)=M_t,
\qquad (M_t\varphi)(t)=t\,\varphi(t).
\]
The operator $M_t$ is bounded, positive, and injective, and
\[
\sigma(M_t)=[0,1].
\]
Hence
\[
\operatorname{FSD}(u)\equiv 0,
\]
and $u$ is maximal on $H$ by Corollary~\ref{cor:constant_levi}.

Define
\[
h(t):=\frac{1}{\sqrt{t}\,|\log t|}\,\mathbf 1_{(0,e^{-1})}(t).
\]
Then $\|h\|_{L^2}=1$ and
\[
\Delta_h(M_t)=\exp\bigl(\langle (\log M_t)h,h\rangle\bigr)
=\exp\!\left(\int_0^{e^{-1}}\log t\,\frac{dt}{t(\log t)^2}\right)=0,
\]
since $\int_0^{e^{-1}}\frac{dt}{t\log t}=-\infty$. Again, $M_t$ is injective, so condition $(2)$ may hold even though condition $(4)$ fails.

This example also shows that the phenomenon is not confined to compact Levi operators: unlike Example~\ref{e111}, the operator $M_t$ is not compact.\hfill{$\Box$}
\end{example}

Examples~\ref{e111} and~\ref{eL2} have
$\langle L_u(a)h,h\rangle>0$ for every $a$ and every nonzero $h$,
although $u$ is maximal. Thus injectivity of the Levi operator is
compatible with maximality. Bounded invertibility is excluded by
Theorem~\ref{maximality_implies_zero}.

\begin{example}\label{eEll2Noncompact}
Let $H=\ell^2$, and define a bounded positive sequence $(w_j)$ by
\[
w_{2n-1}:=\frac1n,
\qquad
w_{2n}:=1
\qquad (n\in\N).
\]
Set
\[
u(z):=\sum_{j=1}^{\infty}w_j|z_j|^2+\sum_{j=1}^{\infty}|z_j|^4,
\qquad z=(z_j)\in H.
\]
The quartic sum is a continuous real polynomial of
degree four on $\ell^2$. Indeed, H\"older's inequality bounds the
real four-linear form
\[
(h_1,h_2,h_3,h_4)\longmapsto
\operatorname{Re}\sum_{j=1}^{\infty}
h_{1,j}\overline{h_{2,j}}h_{3,j}\overline{h_{4,j}}
\]
in absolute value by $\prod_{\nu=1}^{4}\|h_\nu\|$.
The quartic sum is therefore Fr\'echet smooth.
Then
\[
u\in \mathcal{PSH}(H)\cap C^{\infty}(H),
\]
and
\[
L_u(z)=\mathrm{diag}\!\bigl(w_j+4|z_j|^2\bigr)
\qquad (z\in H).
\]
In particular, $L_u(z)$ is positive for every $z\in H$.

Since
\[
w_j+4|z_j|^2>0
\qquad (j\in\N),
\]
the operator $L_u(z)$ is injective for every $z\in H$.

On the other hand, since $z_j\to 0$, we have
\[
w_{2n-1}+4|z_{2n-1}|^2
=
\frac1n+4|z_{2n-1}|^2
\longrightarrow 0.
\]
Hence
\[
\inf \sigma\bigl(L_u(z)\bigr)=0,
\qquad
\operatorname{FSD}(u)(z)=0
\qquad (z\in H)
\]
by Proposition~\ref{prop:extrema}.

At the same time,
\[
w_{2n}+4|z_{2n}|^2
=
1+4|z_{2n}|^2
\longrightarrow 1,
\]
so $L_u(z)$ is not compact for any $z\in H$.

To prove maximality, set
\[
x_n:=\frac1{\sqrt n}\sum_{k=1}^n e_{2k-1},
\qquad H_n:=\sum_{k=1}^n\frac1k.
\]
These are unit vectors, and
\[
\langle L_u(z)x_n,x_n\rangle
=\frac1n\sum_{k=1}^n
  \left(\frac1k+4|z_{2k-1}|^2\right)
\le\frac{H_n}{n}+\frac4n\|z\|^2.
\]
Since $H_n/n\to0$, the right-hand side tends to zero uniformly on
bounded subsets of $H$. Theorem~\ref{prop:approx_null_maximality}
therefore gives maximality.

Thus $u$ is a maximal plurisubharmonic function on $\ell^2$ whose Levi operator is nonconstant, injective, and noncompact at every point. \hfill{$\Box$}
\end{example}

\begin{example}\label{e113}
Let $H$ be infinite-dimensional and let $A\in\mathcal B(H)$ be a
positive, compact, injective operator. The function
\[
u(z):=\log\bigl(1+\langle Az,z\rangle\bigr)
\]
is smooth and plurisubharmonic by Proposition~\ref{prop:phiA_factory},
applied with $\Phi(t)=\log(1+t)$. Indeed,
$\Phi'(t)+t\Phi''(t)=(1+t)^{-2}>0$. With
$q(z)=\langle Az,z\rangle$ and $R_zh=\langle h,Az\rangle Az$,
\[
L_u(z)=\frac{A}{1+q(z)}-\frac{R_z}{(1+q(z))^2}.
\]
This operator is compact. The inequality $0\le R_z\le q(z)A$ gives
\[
\frac{A}{(1+q(z))^2}\le L_u(z)\le\frac{A}{1+q(z)}\le A.
\]
Hence $L_u(z)$ is injective at every point. Since
$\inf\sigma(A)=0$, Proposition~\ref{prop:phiA_factory} gives
maximality and $\operatorname{FSD}(u)\equiv0$. The upper bound also
shows that any sequence of unit vectors $(x_n)$ with
$\langle Ax_n,x_n\rangle\to0$ satisfies the hypothesis of
Theorem~\ref{prop:approx_null_maximality}. Thus this example has
nonconstant, compact, injective Levi operators.
\hfill{$\Box$}
\end{example}

\begin{example}\label{ez4}
The quartic sum in Example~\ref{eEll2Noncompact} also gives a simpler
example. On $H=\ell^2$, put
\[
u(z):=\sum_{j=1}^{\infty}|z_j|^4.
\]
As shown there, $u$ is a continuous real polynomial and is therefore
smooth. Its Levi operator is
\[
L_u(z)=\operatorname{diag}(4|z_j|^2).
\]
It is positive and compact for every $z$, so $u$ is plurisubharmonic
and maximal by Corollary~\ref{cor:pointwise_levi_maximality}.
The uniform criterion also applies. For
$x_n=n^{-1/2}\sum_{j=1}^ne_j$, one has
\[
\langle L_u(z)x_n,x_n\rangle
=\frac4n\sum_{j=1}^n|z_j|^2\le\frac4n\|z\|^2.
\]
The Levi operator is nonconstant and vanishes at zero. In particular,
it need not be injective.
\hfill{$\Box$}
\end{example}

\begin{example}\label{eFiniteRankMoving}
Let $H=\ell^2$. Choose a nondecreasing function
$\eta\in C^\infty([0,\infty))$ with $0\le\eta\le1$, with
$\eta(t)=0$ for $t\le1$, and with $\eta(t)=1$ for $t\ge2$. Put
\[
\chi(t):=\int_0^t\eta(s)\,ds,
\qquad a(t):=\eta(t)+t\eta'(t),
\qquad u(z):=\sum_{j=1}^\infty\chi(|z_j|^2).
\]
The series is locally finite. Indeed, for any $z\in\ell^2$, choose
$N$ such that $|z_j|\le1/2$ for $j>N$. If $\|w-z\|<1/2$, then
$|w_j|<1$ for $j>N$, so the corresponding summands vanish.
Consequently $u$ is smooth, and
\[
L_u(z)=\operatorname{diag}(a(|z_j|^2)).
\]
The function $a$ is nonnegative and bounded, and vanishes on $[0,1]$.
Thus $u$ is plurisubharmonic, and each Levi operator has finite rank.
These ranks are not bounded on $H$. In fact, the Levi operator at
$2\sum_{j=1}^Ne_j$ is the orthogonal projection onto
$\operatorname{span}\{e_1,\ldots,e_N\}$ and has rank $N$.

The uniform criterion also applies to this example. If
$M:=\sup_{t\ge0}a(t)$ and $x_n:=n^{-1/2}\sum_{j=1}^ne_j$, then
\[
\langle L_u(z)x_n,x_n\rangle
=\frac1n\sum_{j=1}^na(|z_j|^2)
\le\frac{M}{n}\#\{j:|z_j|>1\}
\le\frac{M}{n}\|z\|^2.
\]
Hence Theorem~\ref{prop:approx_null_maximality} gives maximality.
Since $L_u(2e_j)=P_{\mathbb C e_j}$, the ranges of the Levi
operators are not contained in any fixed proper closed subspace, and
their kernels have no common nonzero vector. The family
$\{L_u(z):\|z\|<3\}$ is not collectively compact, because its
images of the unit ball contain $(e_j)$. Thus the uniform criterion does not
require a common nonzero kernel vector, a fixed proper closed subspace
containing all ranges, or collective compactness.
\hfill{$\Box$}
\end{example}

\begin{example}\label{e:rank_one_no_uniform_null}
On $H=\ell^2$, set
\[
f(z)=\sum_{j=1}^\infty z_j^2,\qquad u(z)=|f(z)|^2.
\]
The bilinear form $\sum_jz_jh_j$ is continuous by Cauchy--Schwarz, so
$f$ is a continuous holomorphic quadratic polynomial and $u$ is smooth
and plurisubharmonic. Its Levi operator satisfies
\[
\langle L_u(z)h,h\rangle=4\left|\sum_{j=1}^\infty z_jh_j\right|^2,
\qquad L_u(z)h=4\langle h,\overline z\rangle\overline z.
\]
It has rank at most one. Thus $u$ is maximal by
Corollary~\ref{cor:pointwise_levi_maximality}.
Alternatively, since
\[
0\le u(z)\le\|z\|^4,
\]
maximality also follows from Theorem~\ref{thm:bounded_converse}.
Nevertheless, for every $R>0$ and every unit vector $h$,
\[
\sup_{\|z\|<R}\langle L_u(z)h,h\rangle=4R^2.
\]
The upper bound follows from Cauchy--Schwarz, and the reverse inequality
follows by taking $z=t\overline h$ with $t\nearrow R$.
Thus the sufficient hypothesis in
Theorem~\ref{prop:approx_null_maximality} is not necessary for maximality,
even for Levi operators of rank at most one.
This example also distinguishes compactness of each operator from
collective compactness. For every $j$,
\[
L_u(e_j)=4P_{\mathbb C e_j},
\]
where $P_{\mathbb C e_j}$ is the orthogonal projection onto
$\mathbb C e_j$. These ranges are not contained in a fixed proper
closed subspace. The operators have no common nonzero vector in their
kernels, since
$\langle L_u(e_j)h,h\rangle=4|h_j|^2$.
Moreover, the family $\{L_u(z):\|z\|<2\}$ is not collectively
compact, since its images of the unit ball contain the sequence
$(4e_j)$, which has no convergent subsequence.

\end{example}

\begin{example}\label{e:unbounded_on_bounded_sets}
The boundedness assumption in Theorem~\ref{thm:bounded_converse}
is not automatic, even when the Levi operator is constant.
On $H=\ell^2$, define
\[
F(z):=\sum_{j=1}^{\infty}z_j^{2j},
\qquad
u(z):=\sum_{j=1}^{\infty}\frac{|z_j|^2}{j}
       +\operatorname{Re}F(z).
\]
We first verify that $F$ is an entire Fr\'echet holomorphic function.
Fix $z\in\ell^2$ and choose $N$ such that $|z_j|\le1/4$
for $j>N$. If $w\in B(z,1/4)$, then $|w_j|<1/2$ for $j>N$.
Writing $F_j(w)=w_j^{2j}$, for every fixed integer $k\ge0$ we have
\[
\sup_{w\in B(z,1/4)}\|D^kF_j(w)\|
\le (2j)^k2^{k-2j}\qquad(j>N),
\]
where the derivative is zero if $k>2j$.
The series of these bounds converges. Thus the series defining $F$
and the series of its derivatives of every order converge uniformly
on this neighborhood. This proves the asserted holomorphicity.

Since $\operatorname{Re}F$ is pluriharmonic, $u$ is smooth and
plurisubharmonic, with
\[
L_u(z)=\operatorname{diag}(1/j)\qquad(z\in\ell^2).
\]
This operator is compact and injective, and the infimum of its
spectrum is zero. Corollary~\ref{cor:constant_levi} therefore shows
that $u$ is maximal on $\ell^2$.
Nevertheless,
\[
u(2e_j)=\frac4j+4^j\longrightarrow+\infty,
\qquad
u\bigl(2e^{i\pi/(2j)}e_j\bigr)
=\frac4j-4^j\longrightarrow-\infty.
\]
Both sequences lie in $B(0,3)$. Hence $u$ is unbounded above and
below on this bounded ball, although its Levi operator is constant.
\hfill{$\Box$}
\end{example}

\section{Comparison principle}

We apply Proposition~\ref{prop:mixed_add_mult} to the Levi operators
of two plurisubharmonic functions. This gives the multiplicative and
additive comparison principles as the two endpoints of one family
of inequalities.

\begin{theorem}\label{cp_mixed}
Let $\Omega\subset H$ be a bounded domain, and let
$u,v\in\mathcal{PSH}(\Omega)\cap C^2(\Omega)\cap C(\overline\Omega)$.
Suppose that
\[
mI\le L_v(z)\le MI\qquad(z\in\Omega)
\]
for some $0<m\le M$, and that
\[
\Delta_x(L_v(z))\le\Delta_x(L_u(z))
\qquad(z\in\Omega,\ \|x\|=1).
\]
Put $S:=S(M/m)$. If $m<M$, define
\[
L:=\frac{M-m}{\log M-\log m},
\qquad d_c:=L\log(S/c)\quad(1\le c\le S).
\]
If $m=M$, then $S=1$ and we set $d_1:=0$.
For every $c\in[1,S]$ and every $z_*\in H$, the function
\[
w_c(z):=c\,u(z)+d_c\|z-z_*\|^2-v(z)
\]
is plurisubharmonic on $\Omega$.
If $w_c\le0$ on $\partial\Omega$, then $w_c\le0$ on $\Omega$.
\end{theorem}

\begin{proof}
Fix $z\in\Omega$ and put $A=L_u(z)$ and $B=L_v(z)$.
Since $B\ge mI$, Proposition~\ref{prop:extrema} gives
\[
\Delta_x(A)\ge\Delta_x(B)\ge m
\qquad(\|x\|=1).
\]
The same proposition implies $\inf\sigma(A)\ge m$, so $A\ge mI$.
Both operators are therefore positive and invertible. By
Proposition~\ref{log}, their determinant inequalities give $A\gg B$.
For $m<M$, Proposition~\ref{prop:mixed_add_mult}, with $p=1$, yields
\[
B\le cA+d_cI\qquad(1\le c\le S).
\]
For $m=M$, we have $B=mI\le A$, which is the same inequality
with $c=1$ and $d_1=0$.
Since $L_{\|\cdot-z_*\|^2}=I$, it follows that
\[
L_{w_c}(z)=cA+d_cI-B\ge0.
\]
Thus $w_c$ is plurisubharmonic. It is continuous on
$\overline\Omega$, so its boundary inequality and
Lemma~\ref{lem:boundary_max} give the conclusion.
\end{proof}

The determinant inequality in Theorem~\ref{cp_mixed} is required for
every unit vector. An inequality for one fixed vector does not imply
$\log A\ge\log B$, even when $A$ and $B$ commute.

\begin{corollary}\label{cp1}\label{cp2}
Under the hypotheses of Theorem~\ref{cp_mixed}, the functions
\[
S(M/m)u-v
\qquad\text{and}\qquad
u+C(m,M)\|z\|^2-v
\]
are plurisubharmonic on $\Omega$, where $C(m,M)$ is defined in
Proposition~\ref{prop:additive_reverse}.
Each is nonpositive on $\Omega$ if it is nonpositive on
$\partial\Omega$.
\end{corollary}
\begin{proof}
Take $z_*=0$ in Theorem~\ref{cp_mixed} and use $d_S=0$ and
$d_1=C(m,M)$. These identities include $m=M$, when $S=1$ and
$C(m,m)=0$.
\end{proof}

For $m<M$, the additive constant satisfies $0<C(m,M)<M-m$.
Indeed, with $L=(M-m)/(\log M-\log m)$ and $\ell=L/m>1$,
\[
C(m,M)=m+L\bigl(\log(L/m)-1\bigr)
=m(\ell\log\ell+1-\ell)>0.
\]
The last inequality follows because the function
$\ell\mapsto\ell\log\ell+1-\ell$ vanishes at $1$ and has positive
derivative for $\ell>1$.
Also $L<M$, so
\[
C(m,M)<m+L\log(M/m)-L=M-L<M-m.
\]
The hypotheses imply $L_u\ge mI$ and $L_v\le MI$.
These two bounds alone give
$L_v\le L_u+(M-m)I$, whereas
Corollary~\ref{cp2} uses the smaller coefficient $C(m,M)$.

Comparisons with $\|z\|^2$ have a simpler form. By
Proposition~\ref{prop:extrema}, for a positive operator $A$,
\[
\Delta_x(A)\ge1\ \text{for all unit }x
\quad\Longleftrightarrow\quad A\ge I,
\]
and
\[
\Delta_x(A)\le1\ \text{for all unit }x
\quad\Longleftrightarrow\quad A\le I.
\]

\begin{corollary}\label{cp3}\label{cp4}
Let $\Omega\subset H$ be a bounded domain, and let
$u\in\mathcal{PSH}(\Omega)\cap C^2(\Omega)\cap C(\overline\Omega)$.
\begin{enumerate}\itemsep2mm
\item If $\Delta_x(L_u(z))\ge1$ for every $z\in\Omega$ and
 every unit vector $x$, and $u\le\|z\|^2$ on $\partial\Omega$,
 then $u-\|z\|^2$ is plurisubharmonic and nonpositive on $\Omega$.
\item If $\Delta_x(L_u(z))\le1$ for every $z\in\Omega$ and
 every unit vector $x$, and $u\ge\|z\|^2$ on $\partial\Omega$,
 then $\|z\|^2-u$ is plurisubharmonic and nonpositive on $\Omega$.
\end{enumerate}
\end{corollary}
\begin{proof}
The preceding equivalences give $L_u\ge I$ in \textup{(1)} and
$L_u\le I$ in \textup{(2)}. Since $L_{\|\cdot\|^2}=I$, the stated
differences are plurisubharmonic. Their boundary inequalities and
Lemma~\ref{lem:boundary_max} give the conclusions.
\end{proof}

\begin{corollary}\label{cor_bounds}
Let $\Omega\subset H$ be a bounded domain and let
$u\in\mathcal{PSH}(\Omega)\cap C^2(\Omega)\cap C(\overline\Omega)$
be bounded. Assume that
\[
mI\leq L_u(z)\leq MI\qquad\text{for all }z\in \Omega,
\]
for some $0<m\leq M$. Set
$R^2:=\sup_{w\in\partial\Omega}\|w\|^2$ and
$r^2:=\inf_{w\in\partial\Omega}\|w\|^2$.
Then for all $z\in\Omega$,
\begin{align*}
M(\|z\|^2-R^2)+\inf_{w\in\partial\Omega}u(w)
&\le u(z),\\
u(z)&\le m(\|z\|^2-r^2)+\sup_{w\in\partial\Omega}u(w).
\end{align*}
\end{corollary}
\begin{proof}
Since $L_u\geq mI$, the function $u-m\|z\|^2$ is plurisubharmonic.
By Lemma~\ref{lem:boundary_max},
\[
u(z)-m\|z\|^2\leq \sup_{w\in\partial\Omega}\bigl(u(w)-m\|w\|^2\bigr)
\leq \sup_{w\in\partial\Omega}u(w)-mr^2,
\]
which yields the upper bound. Similarly, $L_u\leq MI$ implies
that $M\|z\|^2-u$ is plurisubharmonic, and
Lemma~\ref{lem:boundary_max} gives
\[
M\|z\|^2-u(z)\leq \sup_{w\in\partial\Omega}\bigl(M\|w\|^2-u(w)\bigr)
\leq MR^2-\inf_{w\in\partial\Omega}u(w),
\]
which gives the lower bound.
\end{proof}

\section{Open problems and future directions}\label{sec:open}

We conclude with several questions suggested by the preceding sections and some possible directions for further work.

\begin{question}\label{q:local_global}
\textbf{From local maximality to maximality.}
Let $\Omega$ be a domain in an infinite-dimensional complex Hilbert
space and let $u\in\mathcal{PSH}(\Omega)\cap C^2(\Omega)$. By
Corollary~\ref{cor:local_maximality}, the condition
\[
\operatorname{FSD}(u)\equiv0
\]
is equivalent to local maximality. Does this condition imply
maximality in $\Omega$?

Theorem~\ref{thm:bounded_converse} answers this affirmatively when
$u$ is bounded above and below on the closure of every bounded open
set $G$ with $\overline G\subset\Omega$. Can this boundedness
assumption be removed? If not, can one construct a locally maximal
function and a bounded open set $G$ with $\overline G\subset\Omega$
on which the comparison property fails?

If comparison fails on a bounded open set $G$ with
$\overline G\subset\Omega$, then $u$ cannot be bounded above and below
on $\overline G$, by the proof of Theorem~\ref{thm:bounded_converse}.
Moreover, Proposition~\ref{prop:simple_zero_obstruction} gives a point
of $G$ where zero is an isolated eigenvalue of the Levi operator
of complex multiplicity one. Thus a counterexample must satisfy both
of these restrictions. Neither failure of the condition in
Theorem~\ref{prop:approx_null_maximality} nor unboundedness on a bounded
set alone prevents maximality,
as Examples~\ref{e:rank_one_no_uniform_null} and~\ref{e:unbounded_on_bounded_sets}
show.
\end{question}

\begin{question}\label{q:rough_maximality}
\textbf{Beyond the $C^2$ setting.}
The local characterization in Corollary~\ref{cor:local_maximality}
concerns twice continuously differentiable plurisubharmonic functions.
In finite dimensions, viscosity formulations of the least eigenvalue
of the complex Hessian have been studied by Abja~\cite{A23}.
For continuous plurisubharmonic functions on a Hilbert-space domain,
does a viscosity formulation of
\[
\inf\sigma\bigl(L_u\bigr)=0
\]
characterize local maximality? Under which additional assumptions
does such a formulation give comparison on arbitrary bounded open
sets whose closures lie in the domain?

More generally, can the determinant density, or its vanishing
condition, be interpreted in an approximation or variational framework
for rougher plurisubharmonic functions? Which comparison principles
of Section~6 remain valid in such a framework?
\end{question}

\section*{Acknowledgments} The first- and second-named authors were partially funded by the Jagiellonian University in Kraków under the “Excellence Initiative - Research University” program. The first-named author was partially funded by the Swedish Research Council under grant agreement no.~2025-05053 and by the Swedish Energy Agency under project no.~P2025-04323. The second-named author was also partially funded by the National Science Centre, Poland, under the Weave-UNISONO programme, grant no.~UMO-2025/07/Y/ST1/00146. The fourth-named author was supported in part by the Engineering and Physical Sciences Research Council (EPSRC) grant EP/Y008375/1.

\end{document}